\documentclass[reqno,11pt]{amsart}
\usepackage{amsmath,amssymb,amsthm,mathrsfs}
\usepackage{epsfig,color}
\usepackage[latin1]{inputenc}
\usepackage[english]{babel}
\usepackage[numbers]{natbib}
\usepackage[colorlinks, citecolor=blue, linkcolor=red]{hyperref}

\numberwithin{equation}{section}

\newtheorem{ackn}{Acknowledgments\!}

\def\00{{\bf 0}}

\def\RR{\mathbb R}

\newcommand{\diver}{{\rm div \,}}

\newtheorem*{theorem*}{Theorem}
\newtheorem{theorem}{Theorem}[section]
\newtheorem{lemma}[theorem]{Lemma}
\newtheorem{proposition}[theorem]{Proposition}

\newtheorem{corollary}[theorem]{Corollary}

\theoremstyle{definition}
\newtheorem{example}[theorem]{Example}
\newtheorem{remark}[theorem]{Remark}

\begin{document}
    \title[Quasilinear Liouville equation on manifolds with nonnegative Ricci curvature]{Quasilinear Liouville equation \\ on manifolds with nonnegative Ricci curvature}
  \date{}

\author{Giovanni Catino, Dario Daniele Monticelli, Alberto Roncoroni}

\address{G. Catino, Dipartimento di Matematica, Politecnico di Milano, Piazza Leonardo da Vinci 32, 20133, Milano, Italy.}
\email{giovanni.catino@polimi.it}
\address{D.D. Monticelli, Dipartimento di Matematica, Politecnico di Milano, Piazza Leonardo da Vinci 32, 20133, Milano, Italy.}
\email{dario.monticelli@polimi.it}
\address{A. Roncoroni, Dipartimento di Matematica, Politecnico di Milano, Piazza Leonardo da Vinci 32, 20133, Milano, Italy.}
\email{alberto.roncoroni@polimi.it}

\begin{abstract}
We prove rigidity and classification results for the quasilinear Liouville
equation associated with the $n$-Laplacian on complete noncompact Riemannian
manifolds with nonnegative Ricci curvature. Our first result shows that, under
a sharp logarithmic lower bound, the ambient manifold must be
isometric to the Euclidean space and the solution must be one of the standard
 bubbles. We also prove a finite-mass rigidity theorem under the
corresponding sharp asymptotic lower bound. We show that any logarithmic lower bound forces positive
asymptotic volume ratio and one-endedness of the manifold. Finally, we construct solutions on
nonflat manifolds with nonnegative Ricci curvature showing the sharpness of our hypotheses.
\end{abstract}

\maketitle

\begin{center}


\medskip

\centerline{\bf AMS subject classification: 35J92, 35B53, 53C21.}

\end{center}

\

\

\section{Introduction}

The classical \emph{Liouville equation}
\begin{equation}\label{eq_1}
-\Delta u=e^{u}
\qquad\text{in }\mathbb R^2
\end{equation}
is one of the basic equations in geometric analysis. It appears naturally in
conformal geometry, since conformal changes of the Euclidean metric in dimension
two are governed by scalar elliptic equations of Liouville type. More precisely,
if $g_u=e^u g_{\mathrm{eucl}}$, then the Gaussian curvature of $g_u$ is
prescribed by an equation of the above form, up to normalization. The
classification of entire solutions is therefore closely related to the global
geometry of conformal metrics in $\RR^2$.

Without additional assumptions, the Liouville equation \eqref{eq_1} admits many entire
solutions. The fundamental rigidity statement is the theorem by
Chen-Li~\cite{ChenLi}, which asserts that every solution of \eqref{eq_1} satisfying the \emph{finite-mass condition}
$$
\int_{\mathbb R^2}e^u\,dx<+\infty
$$
is a standard bubble. In our normalization, this means that
$$
u(x)=
\log\left(
\frac{2\sqrt2\,\lambda}
{1+\lambda^2|x-x_0|^2}
\right)^2
$$
for some $\lambda>0$ and $x_0\in\mathbb R^2$. In dimensions $n\ge3$, the classical Laplacian is naturally associated with
the critical semilinear equation
$$
-\Delta u=u^{\frac{n+2}{n-2}}
\qquad\text{in }\mathbb R^n,
$$
whose positive solutions were classified by Caffarelli-Gidas-Spruck~\cite{CaffarelliGidasSpruck} (see also \cite{ChenLi, LiZhang}). The quasilinear
analogue involving the $p-$Laplacian,
$$
-\Delta_p u=u^{p^*-1},\qquad1<p<n,
\qquad
p^*=\frac{np}{n-p},
$$
has also been extensively studied. Classification and symmetry results for
critical $p-$Laplace equations in $\RR^n$, or on Riemannian manifolds $(M,g)$ with nonnegative Ricci curvature, were obtained in
\cite{CM, CMR, CFP, CFR,DamascelliMerchanMontoroSciunzi, FMM, Oup, Sciunzi, Sun-Wang, Vetois, Vetois2}.
Related rigidity phenomena have also been investigated beyond the Riemannian critical Sobolev setting. For singular or degenerate critical
equations arising from Caffarelli--Kohn--Nirenberg inequalities, see e.g.
\cite{CC,CatMonRon2, CirPol}. In CR and
sub-Riemannian geometries, Liouville-type rigidity and classification results
were obtained in the Heisenberg group in \cite{CatLiMonRon, FlyVet,JL}, and on
Sasakian manifolds in \cite{CatMonRonWang}.

The conformally invariant quasilinear analogue of the two-dimensional Liouville
equation on a Riemannian manifolds $(M,g)$ is the \emph{$n-$Liouville equation}
\begin{equation}\label{eq_crit}
	-\Delta_n u=e^{u} \quad \text{ in } M\,  ,
\end{equation}
where
\[
\Delta_nu=\operatorname{div}\left(|\nabla u|^{n-2}\nabla u\right)
\]
is the $n-$Laplace-Beltrami operator. Here a solution of \eqref{eq_crit} is intended in the weak sense, i.e. $u\in W^{1,n}_{loc}(M)\cap L^{\infty}_{loc}(M)$ which satisfies
$$
\int_{M} |\nabla u|^{n-2}g(\nabla u,\nabla\varphi)\, dV_g=\int_{M} e^u\varphi\, dV_g\, , \quad \text{ for all } \varphi\in W^{1,n}_{0}(M)\,,
$$
where $W^{1,n}_0(M)$ denotes the set of compactly supported functions of $W^{1,n}(M)$.

In this setting, Esposito~\cite{Esp} proved the Euclidean classification
theorem: every solution of
$$
-\Delta_nu=e^u
\qquad\text{in }\mathbb R^n
$$
satisfying the \emph{finite-mass condition}
$$
\int_{\RR^n}e^u\,dx<\infty
$$
is one of the explicit bubbles
\begin{equation}\label{Tal_n}
\mathcal{U}_{\lambda,x_0}(x)=\log\left(\frac{c_n\lambda}{1+\lambda^{\frac{n}{n-1}}|x-x_0|^{\frac{n}{n-1}}}\right)^n\, ,
\end{equation}
for some $\lambda>0$ and $x_0\in\mathbb R^n$, where
$$
c_n=n^{1/n} \left( \frac{n^2}{n-1}\right)^{\frac{n-1}{n}}\,.
$$
These solutions satisfy the sharp logarithmic asymptotic behaviour
$$
\mathcal{U}_{\lambda,x_0}(x)
=
-\frac{n^2}{n-1}\log |x|+O(1) \, ,
\quad\text{as } |x|\to+\infty\, .
$$
The same problem has also been studied in connection with blow-up analysis,
Harnack inequalities, anisotropic
extensions and weighted singular analogues; see, for instance,
\cite{CirEspLi,CiraoloLi,EspositoLucia,Sun-Wang}.

The purpose of this paper is to investigate how the Euclidean rigidity theory
for the $n-$Liouville equation extends to complete Riemannian manifolds with
nonnegative Ricci curvature. This is a natural geometric setting: on the one
hand, nonnegative Ricci curvature provides strong global comparison tools; on
the other hand, it still allows many non-Euclidean asymptotic geometries,
including cylindrical ends, asymptotically conical ends, and ends with zero
asymptotic volume ratio.

Several recent works have shown that special solutions of critical elliptic
equations may force strong rigidity of the ambient manifold. Results of this
type were obtained, for instance, by Catino--Monticelli~\cite{CM}
and by Ciraolo--Farina--Polvara~\cite{CFP} for semilinear
equations on manifolds with nonnegative Ricci curvature. In dimension two,
Cai--Lai~\cite{CaiLai} studied Liouville equations on complete surfaces with
nonnegative Gaussian curvature under finite-mass assumptions. More recently,
Ou~\cite{Ou} proved a sharp rigidity theorem for the Liouville equation on
complete noncompact surfaces with nonnegative curvature. In the normalization
$$
-\Delta u=e^u\qquad\text{in }M^2\, ,
$$
Ou's theorem states that if
$$
u(x)\ge -4\log r(x)-\alpha\log F(r(x)),
$$
$\alpha\geq0$, outside a compact set, where $r(x)$ is the Riemannian distance of $x$ from a fixed origin and $F$ is positive, nondecreasing, and satisfies
$$
\int_1^\infty\frac{dt}{tF(t)}=+\infty\, ,
$$
then the surface is isometric to the Euclidean plane and the solution is a
standard bubble. The coefficient \(4\) is sharp, because it is precisely the
logarithmic decay rate of the Euclidean bubbles.

Our first main result is the $n-$dimensional quasilinear analogue of this
theorem.

\begin{theorem}\label{teo1}
	Let $(M^n,g)$, $n\ge2$, be a complete, connected, noncompact Riemannian
manifold with nonnegative Ricci curvature. Let $u$ be a weak solution of
	$$
	-\Delta_n u=e^u \quad \text{ in } M\,  ,
	$$
	such that
	$$
	u(x)\geq -\frac{n^2}{n-1}\log r(x) - \alpha\log F(r(x))\, , \quad  \text{ for } r(x)>1\,,
	$$
	and some arbitrary $\alpha\geq 0$, where $F(t)$ is a positive, nondecreasing function satisfying
	\begin{equation}\label{1}
		\int_1^{+\infty} \frac{1}{t F(t)}\, dt=+\infty.
	\end{equation}
	Then $(M,g)$ is isometric to $\mathbb{R}^n$ with the Euclidean metric and $u$ is given by \eqref{Tal_n}.
\end{theorem}

Thus the Euclidean logarithmic decay rate is not only the asymptotic behaviour
of the explicit solutions, but also the sharp rigidity threshold on complete
manifolds with nonnegative Ricci curvature. The coefficient $\frac{n^2}{n-1}$ is optimal. More precisely, for every $\delta>\frac{n^2}{n-1}$, we construct complete nonflat rotationally symmetric manifolds with
nonnegative Ricci curvature supporting radial solutions whose logarithmic decay has coefficient strictly between $\frac{n^2}{n-1}$ and $\delta$. Therefore
the leading coefficient in the lower bound cannot be replaced by any larger
one. This result improves the result in \cite{Sun-Wang}.

We also prove that logarithmic lower bounds have an intrinsic geometric
consequence. We denote by
$$
\operatorname{AVR}(M,g)
:=
\lim_{R\to+\infty}
\frac{\operatorname{Vol}(B_R(o))}{\omega_nR^n}
$$
the \emph{asymptotic volume ratio} of $M$, where $B_R(o)$ denotes the geodesic ball of radius $R$ centered at $o\in M$ and $\omega_{n}:=|\mathbb B^{n}_1|$ is the volume of the unit ball in the Euclidean space. By Bishop--Gromov theorem, this limit
exists and is independent of the base point $o$. We also let $\sigma_{n-1}:=|\mathbb S^{n-1}|$ be the $(n-1)$-dimensional Hausdorff measure of the unit sphere in the Euclidean space, so that $n\omega_n=\sigma_{n-1}$.

\begin{proposition}\label{prop_log_lb}
Let $(M^n,g)$, $n\ge2$, be a complete, connected, noncompact Riemannian
manifold with nonnegative Ricci curvature. Let $u$ be a weak solution of
$$
-\Delta_n u=e^u\qquad\text{in }M\, .
$$
Assume that there exist constants $\beta>0$ and $R_0>0$ such that
\begin{equation}\label{41}
u(x)\ge -\beta\log r(x)\qquad\text{for }r(x)\ge R_0,
\end{equation}
for some fixed point $o\in M$. Then
$$
\theta:=\operatorname{AVR}(M,g)>0.
$$
Moreover, $M$ has only one end.
\end{proposition}

Under the hypotheses of Proposition \ref{prop_log_lb} by Bishop-Gromov volume comparison theorem we have
$$
\theta\omega_nR^n\leq\operatorname{Vol}(B_R(o))\leq\omega_nR^n\, ,
$$
for every $o\in M$ and  $R>0$. In dimension two, Proposition \ref{prop_log_lb} recovers the conformal-type conclusion of \cite[Theorem 1.2]{Ou}. Indeed, by the classical results by Cohn--Vossen \cite{CoVo} and Huber \cite{Hub}, a complete noncompact surface with nonnegative Gaussian curvature is either conformally equivalent to the Euclidean plane or belongs to one of the flat linear-growth alternatives, namely the flat cylinder and, in the nonorientable case, the flat open M\"{o}bius strip. The positive asymptotic volume ratio given by Proposition \ref{prop_log_lb} rules out these flat
linear-growth alternatives. Hence the surface is conformally equivalent to $(\mathbb R^2,g_{\mathrm{eucl}})$; see also \cite[Theorem~1.2]{Ou}.
In dimensions $n\ge3$, no analogous conformal classification follows from $\operatorname{Ric}\ge0$ alone. Thus Proposition \ref{prop_log_lb} only gives
a volume-growth and the one-endedness of $M$ in higher dimension.

Our third result is a \emph{finite-mass rigidity theorem.}

\begin{theorem}\label{teo2}
Let $(M^n,g)$, $n\ge2$, be a complete, connected, noncompact Riemannian
manifold with nonnegative Ricci curvature. Let $u$ be a weak solution of
$$
-\Delta_n u=e^u\qquad\text{in }M,
$$
such that
\begin{equation}\label{40}
\mathcal{M}:=\int_M e^u\,dV_g<+\infty.
\end{equation}
Assume moreover that
\begin{equation}\label{37}
u(x)\ge-\frac{n^2}{n-1}\log r(x)+o(\log r(x))\qquad\text{as }r(x)\to+\infty\,.
\end{equation}
Then $(M,g)$ is isometric to $\mathbb R^n$ with the Euclidean metric and $u$ is given by \eqref{Tal_n}.
\end{theorem}

This result should be viewed as a curved analogue of Esposito's Euclidean
classification~\cite{Esp}, but with an essential additional asymptotic condition. Indeed, unlike in the Euclidean case, finite mass alone does not imply rigidity on
curved manifolds, even under nonnegative Ricci curvature and positive
asymptotic volume ratio. The asymptotically conical Example \ref{ex_sharpness} constructed in Section~\ref{sec3} have positive asymptotic volume ratio and carry
finite-mass solutions, but the underlying manifolds are nonflat. Thus the sharp lower bound in Theorem~\ref{teo2} is necessary: it selects the Euclidean logarithmic rate and excludes nonflat conical models.

The logarithmic lower bound in Proposition \ref{prop_log_lb} cannot be dropped
in general. Indeed, on a \emph{product cylinder}
$$
(N^{n-1}\times\mathbb R,g_N+dt^2)\, ,
$$
with $N$ compact and $\operatorname{Ric}_{g_N}\ge0$, one can construct
finite-mass solutions of
$$
-\Delta_n u=e^u
$$
depending only on the $\mathbb R$-variable and satisfying
$$
u(y,t)\sim-a|t|\, , \quad\text{as } |t|\to+\infty\,  ,
$$
for some $a>0$. Thus solutions may exist on cylindrical manifolds with
$\operatorname{Ric}\ge0$, but their decay is faster than logarithmic. The
logarithmic lower bound is precisely what rules out this cylindrical behaviour. See Example \ref{exe2}.

A similar construction in $\mathbb{R}^n$ yields examples of solutions of the $n-$Liouville equation in the Euclidean space with infinite mass that have
non-logarithmic and anisotropic behavior at infinity. See Example \ref{exe3}.

In Example \ref{exe4} we construct complete rotationally symmetric manifolds with nonnegative Ricci curvature and zero asymptotic volume ratio that support finite-mass solutions, which decay only slightly faster than logarithmic.

The proof of Theorem~\ref{teo1} is based on a quasilinear $P-$function method. We introduce the auxiliary function
$$
w=e^{-u/n}\, .
$$
Then $w$ satisfies
$$
\Delta_nw
=
(n-1)\frac{|\nabla w|^n}{w}
+
\frac{1}{n^{n-1}w}
=:G\, .
$$
We associate to $w$ the nonlinear vector field
$$
\mathbf{v}=|\nabla w|^{n-2}\nabla w\,  ,
$$
and consider the traceless part $\mathring{\mathbf V}$ of $\nabla \mathbf{v}$. A
Bochner-type computation, together with a sharp algebraic inequality, yields
the weighted distributional inequality
$$
\operatorname{div}
\left(
G^{-b}w^{1-n}\mathbf{v}\cdot\mathring{\mathbf V}
\right)\ge
(1-b)G^{-b}w^{1-n}|\mathring{\mathbf V}|^2 \, ,  \quad \text{ for every } b\in[0,1)\, .
$$
The freedom in the parameter \(b\) is crucial in order
to treat arbitrary powers of the auxiliary function \(F\). A Karp-type
Liouville theorem~\cite{karp} then forces
$$
\mathring{\mathbf V}\equiv0\, .
$$
The resulting homothetic structure leads to the Euclidean rigidity through a classical theorem of Tashiro~\cite{tashiro}.

The finite-mass Theorem \ref{teo2} is proved by a different argument. The first ingredient
is the sharp isoperimetric inequality on complete manifolds with nonnegative
Ricci curvature and positive asymptotic volume ratio, due to Brendle~\cite{Bre}
and Balogh--Krist\'aly~\cite{BaKr}. The second ingredient is a nonlinear
capacitary estimate for the $n-$Laplacian, which replaces the classical
logarithmic potential estimates available in dimension two. Combining these two
estimates gives a sharp lower bound on the logarithmic decay of finite-mass
solutions. Under the Euclidean critical lower bound, equality must occur in the
sharp isoperimetric inequality, forcing the ambient manifold to be Euclidean.

Finally, we complement the rigidity theorems with examples showing that the
assumptions are sharp. On product cylinders $N^{n-1}\times\mathbb R$, with
$N$ compact and $\operatorname{Ric}_N\ge0$, we construct finite-mass
solutions depending only on the $\mathbb R-$variable and satisfying
$$
u(y,t)\sim -a|t|\, ,
\quad\text{as } |t|\to+\infty.
$$
These solutions decay faster than logarithmically and live on manifolds with
zero asymptotic volume ratio. We also construct rotationally symmetric examples
with zero asymptotic volume ratio carrying finite-mass solutions whose decay is
super-logarithmic but slower than any positive power. These examples show that
the logarithmic scale is the borderline between positive volume growth and
collapsing-volume geometries.

The rest of the paper is organized as follows. In Section~\ref{sec:diff-ineq}
we derive the weighted differential inequality for the nonlinear $P-$function.
In Section~\ref{sec3} we prove the main rigidity theorem and construct the
sharpness examples for the leading coefficient. In Section~\ref{sec:finite-mass}
we prove the positive asymptotic volume ratio criterion, the finite-mass
rigidity theorem, and the examples with faster-than-logarithmic decay. In Appendix~\ref{app:karp} we record a weighted Karp-type criterion,
adapted to our applications, showing that a locally weighted subsolution
with borderline annular growth has identically vanishing weighted energy.

\

\section{A differential inequality}\label{sec:diff-ineq}

Let $u$ be a solution to \eqref{eq_crit}, we define the auxiliary function
\begin{equation}\label{w}
w=e^{-u/n}\, .
\end{equation}
It is easy to see that $w\in W^{1,n}_{loc}(M)\cap L^{\infty}_{loc}(M)$.
A direct computation shows that
\begin{equation}\label{G}
\Delta_n w=(n-1)\frac{|\nabla w|^n}{w}+ \frac{1}{n^{n-1}w}=:G\, .
\end{equation}
Moreover, we define the following vector field
$$
\mathbf{v}=|\nabla w|^{n-2}\nabla w\, ,
$$
and
$$
\mathbf{V}=\begin{cases}
\nabla \mathbf{v} & \text{in } M\setminus\Omega_{cr}  \\
0 & \text{in } \Omega_{cr}\, ,
\end{cases}
$$
where
$$
\Omega_{cr}:=\{x\in M:\nabla u(x)=0\}=\{x\in M:\nabla w(x)=0\}\, .
$$
We also define
$$
\mathring{\mathbf{V}}=\mathbf{V} - \frac{\mathrm{tr}(\mathbf{V})}{n}g\, ,
$$
the trace-less tensor of $\mathbf{V}$.

The differential inequality below is in the same spirit as the
$P$-function identities used in recent classification results for critical
quasilinear and $n$-Liouville equations; see
\cite{CirEspLi,CFP,Sun-Wang}.

\begin{lemma}
\label{lem_ineq}
Let $(M,g)$ be a Riemannian manifold with $\operatorname{Ric}\ge0$, and let
$u$ be a weak solution of \eqref{eq_crit}. Then, for every $b\in[0,1)$,
$$
\operatorname{div}\left(G^{-b}w^{1-n}\mathbf{v}\cdot\mathring{\mathbf{V}}\right)\geq (1-b)G^{-b}w^{1-n}|\mathring{\mathbf{V}}|^2
$$
in the sense of distributions on $M$.
\end{lemma}

\begin{proof}
By the local regularity theory for \(p\)-Laplace type equations (see  \cite{AnCiFa,DiB,Tol}), any solution of \eqref{eq_crit} satisfies $u\in C^{1,\alpha}_{\mathrm{loc}}(M)$, $u$ is smooth in $M\setminus\Omega_{cr}$, and $|\Omega_{cr}|=0$. Hence the following pointwise computations are justified in $M\setminus\Omega_{cr}$. Moreover, we have
$$
|\nabla w|^{\,n-2}\nabla w \in W^{1,2}_{\mathrm{loc}}(M),
\qquad
|\nabla w|^{\,n-2}\nabla^2 w \in L^2_{\mathrm{loc}}(M),
$$
hence the pointwise inequality on $M\setminus\Omega_{cr}$ extends to $M$ in the sense of distributions. We first note that from the definition of $G$ we have
\begin{equation}\label{3}
\nabla G=-\frac{G}{w}\nabla w + \frac{n}{w}|\nabla w|^{2-n} \mathbf{v} \cdot \mathbf{V}=\frac{n}{w}|\nabla w|^{2-n} \mathbf{v} \cdot \mathring{\mathbf{V}}\, ,
\end{equation}
where we used the fact that
$$
\mathrm{tr}(\mathbf{V})=\mathrm{div}(\mathbf{v})=\Delta_n w=G\, .
$$
Now we compute
\begin{align*}
\mathrm{div}\left(\mathbf{v}\cdot\mathring{\mathbf{V}}\right)=& \mathrm{div}\left(\mathbf{v}\cdot\mathbf{V}\right)-\frac{1}{n}\mathrm{div}\left(\mathrm{tr}(\mathbf{V})\mathbf{v}\right)\\
=& \langle \nabla\left(\mathrm{tr}(\mathbf{V}\right),\mathbf{v}\rangle + \mathrm{Ric}(\mathbf{v},\mathbf{v}) + |\mathbf{V}|^2 - \frac{1}{n} \langle \nabla\left(\mathrm{tr}(\mathbf{V}\right),\mathbf{v}\rangle - \frac{1}{n}\mathrm{tr}(\mathbf{V})^2 \\
=& \frac{n-1}{n}\langle \nabla\left(\mathrm{tr}(\mathbf{V}\right),\mathbf{v}\rangle + \mathrm{Ric}(\mathbf{v},\mathbf{v}) + |\mathring{\mathbf{V}}|^2 \, ,
\end{align*}
where we used the fact that
$$
|\mathring{\mathbf{V}}|^2=|\mathbf{V}|^2 - \frac{1}{n}\mathrm{tr}(\mathbf{V})^2 \, ,
$$
and the Bochner formula
$$
\mathrm{div}\left(\mathbf{v}\cdot\mathbf{V}\right)=\langle \nabla\left(\mathrm{tr}(\mathbf{V}\right),\mathbf{v}\rangle + \mathrm{Ric}(\mathbf{v},\mathbf{v}) + |\mathbf{V}|^2\,  .
$$
Next we  compute
\begin{align*}
\mathrm{div}\left(w^{1-n}\mathbf{v}\cdot\mathring{\mathbf{V}}\right)=&(1-n)w^{-n}\langle\nabla w,\mathbf{v}\cdot\mathring{\mathbf{V}}\rangle+w^{1-n}\mathrm{div}\left(\mathbf{v}\cdot\mathring{\mathbf{V}}\right)  \\
=& (1-n)w^{-n}\langle\nabla w,\mathbf{v}\cdot\mathring{\mathbf{V}}\rangle+\frac{n-1}{n}w^{1-n}\langle\nabla \mathrm{tr}(\mathbf{V}),\mathbf{v}\rangle+ w^{1-n}\mathrm{Ric}(\mathbf{v},\mathbf{v}) +w^{1-n}|\mathring{\mathbf{V}}|^2\, .
\end{align*}
Since
$$
\mathbf{v}\cdot\mathring{\mathbf{V}}=\dfrac{1}{n}\nabla\left(\mathrm{tr}(\mathbf{V})\right)w|\nabla w|^{n-2}
$$
then
$$
\langle\nabla w,\mathbf{v}\cdot\mathring{\mathbf{V}}\rangle=\dfrac{1}{n}w\langle\nabla\left(\mathrm{tr}(\mathbf{V})\right),|\nabla w|^{n-2}\nabla w\rangle=\dfrac{1}{n}w\langle\nabla\left(\mathrm{tr}(\mathbf{V})\right),\mathbf{v}\rangle \, .
$$
Summing up,
\begin{equation}\label{diff_id1}
\mathrm{div}\left(w^{2-n}|\nabla w|^{n-2}\nabla G\right)=n\,\mathrm{div}\left(w^{1-n}\mathbf{v}\cdot\mathring{\mathbf{V}}\right)=n w^{1-n} \mathrm{Ric}(\mathbf{v},\mathbf{v}) +n w^{1-n} |\mathring{\mathbf{V}}|^2\, .
\end{equation}
Now we compute, for $b\in[0,1)$,
$$
\begin{aligned}
\operatorname{div}\left(G^{-b}w^{1-n}\mathbf{v}\cdot\mathring{\mathbf{V}}\right) &=G^{-b}\operatorname{div}\left(w^{1-n}\mathbf{v}\cdot\mathring{\mathbf{V}}\right)
-bG^{-b-1}w^{1-n}\langle\nabla G,\mathbf{v}\cdot\mathring{\mathbf{V}}\rangle \\
&= G^{-b}w^{1-n}\operatorname{Ric}(\mathbf{v},\mathbf{v})+G^{-b}w^{1-n}|\mathring{\mathbf{V}}|^2-bG^{-b-1}w^{1-n} \langle\nabla G,\mathbf{v}\cdot\mathring{\mathbf{V}}\rangle .
\end{aligned}
$$
We have the following algebraic inequality (see \cite[Lemma 2.1]{Sun-Wang} with $p=n$)
$$
| \mathbf{v} \cdot \mathring{\mathbf{V}}|^2\leq \frac{n-1}{n}| \mathbf{v}|^2|\mathring{\mathbf{V}}|^2.
$$
and moreover from the definition of $\mathbf{v}$ it is immediate to see that $|\mathbf{v}|^2=|\nabla w|^{2n-2}$, hence we have
$$
|\langle\nabla G,\mathbf{v}\cdot\mathring{\mathbf{V}}\rangle|\leq\frac{n}{w}|\nabla w|^{2-n}|\mathbf{v} \cdot \mathring{\mathbf{V}}|^2\leq
\frac{n-1}{w}|\nabla w|^{2-n}|\mathbf{v}|^2|\mathring{\mathbf{V}}|^2=\frac{n-1}{w}|\nabla w|^{n}|\mathring{\mathbf{V}}|^2
$$
where we also used \eqref{3}. Therefore
$$
\begin{aligned}
\operatorname{div}\left(G^{-b}w^{1-n}\mathbf{v}\cdot\mathring{\mathbf{V}}\right) &\geq G^{-b}w^{1-n}\operatorname{Ric}(\mathbf{v},\mathbf{v})+G^{-b}w^{1-n}|\mathring{\mathbf{V}}|^2-(n-1)bG^{-b-1}w^{-n}|\nabla w|^n|\mathring{\mathbf{V}}|^2\, .
\end{aligned}
$$
From the definition of $G$ we immediately see that
$$
G^{-1}\leq \frac{1}{n-1}\frac{w}{|\nabla w|^n}\, ,
$$
and thus
$$
\begin{aligned}
\operatorname{div}\left(G^{-b}w^{1-n}\mathbf{v}\cdot\mathring{\mathbf{V}}\right) &\geq G^{-b}w^{1-n}\operatorname{Ric}(\mathbf{v},\mathbf{v})+(1-b)G^{-b}w^{1-n}|\mathring{\mathbf{V}}|^2\\
&\geq (1-b)G^{-b}w^{1-n}|\mathring{\mathbf{V}}|^2.
\end{aligned}
$$
This proves the claim.
\end{proof}

\

\section{Proof of Theorem \ref{teo1}}\label{sec3}

We first record an elementary estimate which will be used in the proof of the
main theorem.

\begin{lemma}\label{lem_est}
Let $(M,g)$ be a complete Riemannian manifold with $\operatorname{Ric}\ge0$,
and let $u$ be a weak solution of \eqref{eq_crit}. Let
$$
w=e^{-u/n}
$$
and
$$
G=(n-1)\frac{|\nabla w|^n}{w}+\frac{1}{n^{n-1}w}.
$$
Then, for every $1<q<n$ and every nonnegative $\eta\in C_c^\infty(M)$, one has
$$
\int_M G w^{1-q}\eta^n\,dV_g\le C_q\int_M w^{n-q}|\nabla\eta|^n\,dV_g .
$$
\end{lemma}

\begin{proof}
Since $\Delta_n w=G$, testing the equation with $w^{1-q}\eta^n$ gives
$$
\begin{aligned}
\int_M G w^{1-q}\eta^n\,dV_g &= \int_M \Delta_n w\, w^{1-q}\eta^n\,dV_g \\
&= (q-1)\int_M |\nabla w|^n w^{-q}\eta^n\,dV_g -n\int_M w^{1-q}\eta^{n-1}|\nabla w|^{n-2}\langle\nabla w,\nabla\eta\rangle\,dV_g .
\end{aligned}
$$
On the other hand, using the definition of $G$,
$$
\int_M G w^{1-q}\eta^n\,dV_g=(n-1)\int_M |\nabla w|^n w^{-q}\eta^n\,dV_g+\frac{1}{n^{n-1}}\int_M w^{-q}\eta^n\,dV_g .
$$
Combining the two identities, we obtain
$$
(n-q)\int_M |\nabla w|^n w^{-q}\eta^n\,dV_g+\frac{1}{n^{n-1}}\int_M w^{-q}\eta^n\,dV_g\leq n\int_M w^{1-q}\eta^{n-1}|\nabla w|^{n-1}|\nabla\eta|\,dV_g .
$$
Since $q<n$, Young's inequality yields
$$
\int_M |\nabla w|^n w^{-q}\eta^n\,dV_g+\int_M w^{-q}\eta^n\,dV_g\leq C_q\int_M w^{n-q}|\nabla\eta|^n\,dV_g .
$$
Using again the expression of $G$, we conclude
$$
\int_M G w^{1-q}\eta^n\,dV_g\leq C_q\int_M w^{n-q}|\nabla\eta|^n\,dV_g .
$$
\end{proof}

\begin{proof}[Proof of Theorem \ref{teo1}]
Without loss of generality we can assume that $F(t)\geq1$ for every t and that $\alpha>\frac{n^2}{2(n-1)}$. Fix $b\in\left(\frac{n-2}{n},1\right)$ to be chosen later. We set
$$
\theta:=2-\frac{2}{n}-b.
$$
Then $\frac{n-2}{n}<\theta<1$. We will apply Lemma \ref{lem_karp} with
$$
T:=G^{-b}w^{1-n}\mathbf{v}\cdot\mathring{\mathbf{V}},\qquad E:=(1-b)G^{-b}w^{1-n}|\mathring{\mathbf{V}}|^2, \qquad
h:=G^{-b}w^{1-n}|\nabla w|^{2n-2}.
$$
By the differential inequality of Lemma \ref{lem_ineq}
$$
\operatorname{div}T=\operatorname{div}\left(G^{-b}w^{1-n}\mathbf{v}\cdot\mathring{\mathbf{V}}\right)\geq (1-b)G^{-b}w^{1-n}|\mathring{\mathbf{V}}|^2=E
$$
in the sense of distributions. Moreover, since $|\mathbf{v}|^2=|\nabla w|^{2n-2}$, we have
$$
|T|^2=G^{-2b}w^{2-2n}|\mathbf{v}\cdot\mathring{\mathbf{V}}|^2\leq G^{-2b}w^{2-2n}|\mathbf{v}|^2|\mathring{\mathbf{V}}|^2=G^{-2b}w^{2-2n}|\nabla w|^{2n-2}|\mathring{\mathbf{V}}|^2
=\frac{1}{1-b}hE\,.
$$
Thus the structural assumptions of Lemma \ref{lem_karp} are satisfied. It remains to prove a suitable growth estimate for $h$. Since
$G\geq (n-1)\frac{|\nabla w|^n}{w}$, we have $|\nabla w|^n\le CwG$. Consequently,
$$
h = G^{-b}w^{1-n}|\nabla w|^{2n-2} \leq C G^{2-\frac2n-b}w^{3-n-\frac2n}=C G^\theta w^{3-n-\frac2n}.
$$
We now fix any $q\in(1+\frac{(n-1)(n-2)}{n\theta},n)$. Define
$$
m:=\frac{3-n-\frac2n+\theta(q-1)}{1-\theta}.
$$
By the choice of $q$, we have $m>0$. Hence, by Holder's inequality,
\begin{equation}\label{13}
\begin{aligned}
\int_{B_R} h\,dV_g&\leq C\int_{B_R}G^\theta w^{3-n-\frac2n}\,dV_g  \\
&= C\int_{B_R} \left(Gw^{1-q}\right)^\theta w^{3-n-\frac2n-\theta(1-q)}\,dV_g  \\
&\leq C \left(\int_{B_R}Gw^{1-q}\,dV_g\right)^\theta \left(\int_{B_R}w^m\,dV_g\right)^{1-\theta}.
\end{aligned}
\end{equation}
We use Lemma \ref{lem_est} with a family of standard cut-off functions $\eta\in C^2(M)$ such that
$$
\eta\equiv 1   \text{ in } B_R(o)\, , \qquad \eta\equiv 0  \text{ in } M\setminus B_{2R}(o)\, , \qquad  0\leq \eta \leq 1 \quad \text{ in } M\, ,
$$
and
$$
|\nabla\eta|\leq \frac{C}{R} \quad \text{ in } A_R:=B_{2R}(o)\setminus B_R(o)\, ,
$$
for some $C>0$ and every $R\gg1$. Then we have
\begin{equation}\label{12}
\int_{B_R}Gw^{1-q}\,dV_g\leq\frac{C_q}{R^n}\int_{B_{2R}\setminus B_R}w^{n-q}\,dV_g .
\end{equation}
Since by our assumptions
$$
u(x)\ge-\frac{n^2}{n-1}\log r(x)-\alpha\log F(r(x))
$$
for $r(x)>1$, and $w=e^{-u/n}$, we have
$$
w(x)\leq C r(x)^{\frac{n}{n-1}}F(r(x))^{\frac{\alpha}{n}}
$$
for $r(x)>1$. Thus, for every $s\geq0$, by Bishop-Gromov volume comparison,
$$
\int_{B_R}w^s\,dV_g \leq C_s R^{n+\frac{ns}{n-1}}F(R)^{\frac{\alpha s}{n}}
$$
for every $R\gg1$. Applying this estimate with $s=n-q$ and with $s=m$, from \eqref{12} we get
$$
\int_{B_R}Gw^{1-q}\,dV_g \leq C R^{\frac{n(n-q)}{n-1}} F(R)^{\frac{\alpha(n-q)}{n}},
$$
and
$$
\int_{B_R}w^m\,dV_g\leq C R^{n+\frac{nm}{n-1}}F(R)^{\frac{\alpha m}{n}}.
$$
Therefore from \eqref{13}
$$
\begin{aligned}
\int_{B_R}h\,dV_g &\leq C R^{\theta\frac{n(n-q)}{n-1}+(1-\theta)\left(n+\frac{nm}{n-1}\right)}F(R)^{\theta\frac{\alpha(n-q)}{n}+(1-\theta)\frac{\alpha m}{n}}\\
&=C R^2 F(R)^{\delta_b}\,,
\end{aligned}
$$
with $\delta_b:=\frac{\alpha(n-1)(1-b)}{n}$. We now choose
$$
b=1-\frac{n}{\alpha(n-1)}\,,
$$
so that $\delta_b=1$. Note that our initial assumption $\alpha>\frac{n^2}{2(n-1)}$ implies that $b\in(\frac{n-2}{n},1)$ as required.
Then we have
$$
\int_{B_R}h\,dV_g \leq C R^2 F(R)\,,
$$
with $F$ positive, nondecreasing and satisfying \eqref{1}. Thus condition \eqref{8} in Lemma \ref{lem_karp} is also satisfied.
We conclude by Lemma \ref{lem_karp} that $$E=(1-b)G^{-b}w^{1-n}|\mathring{\mathbf{V}}|^2\equiv0\,.$$
Thus
\begin{equation}\label{2}
\mathring{\mathbf{V}} = \nabla  \mathbf{v}-\frac{\diver\mathbf{v}}{n}g \equiv 0
\end{equation}
on $M\setminus\Omega_{\rm cr}$.
From
$$
\nabla G=\frac{n}{w}|\nabla w|^{2-n}\mathbf v\cdot\mathring{\mathbf V}
$$
we get
$$
        \nabla G\equiv0
        \qquad\text{on }M\setminus\Omega_{\rm cr}.
$$
Moreover, using
$$
\operatorname{div}\left(w^{2-n}|\nabla w|^{n-2}\nabla G\right)
=
nw^{1-n}\operatorname{Ric}(\mathbf v,\mathbf v)
+
nw^{1-n}|\mathring{\mathbf V}|^2\, ,
$$
we obtain
$$
        \operatorname{Ric}(\mathbf v,\mathbf v)=0
        \qquad\text{on }M\setminus\Omega_{\rm cr}.
$$
Since $\mathbf v=0$ on $\Omega_{\rm cr}$, this identity holds a.e. on $M$. As $\operatorname{Ric}\ge0$, it follows that
$$
        \operatorname{Ric}_{ij}\mathbf v^j=0
        \quad\text{a.e. on }M.
$$
Since $\mathbf v\in W^{1,2}_{\rm loc}(M)$ and $\mathring{\mathbf V}=0$ a.e. on \(M\), the identity
$$
        \nabla_i\mathbf v_j=\frac{G}{n}g_{ij}
$$
holds weakly, and hence a.e., on $M$. Taking the divergence in the weak sense gives
$$
        \frac1n\nabla_iG
        =
        \nabla_j\nabla_i\mathbf v_j
        =
        \nabla_i\operatorname{div}\mathbf v
        -
        \operatorname{Ric}_{ij}\mathbf v^j
        =
        \nabla_iG \, .
$$
Therefore
$$
        \nabla G=0\, ,
$$
in the sense of distributions on $M$. Since $G$ is continuous and $M$ is connected, $G$ is constant, i.e.
$$
       G\equiv\lambda .
$$
By the explicit formula
$$
        G=(n-1)\frac{|\nabla w|^n}{w}+\frac{1}{n^{n-1}w}\, ,
$$
we have $\lambda>0$. Hence
$$
        \nabla_i\mathbf v_j=\frac{\lambda}{n}g_{ij}
        \qquad\text{weakly on }M.
$$
Equivalently, setting
$$
        X:=\frac n\lambda\mathbf v,
$$
we have
$$
        \nabla_iX_j=g_{ij}
        \quad\text{weakly on }M.
$$
By standard regularity for this first-order linear system, $X$ is smooth and
$$
        \nabla_iX_j=g_{ij}
        \qquad\text{on }M.
$$
Thus $X$ is a concurrent vector field on the complete manifold $M$. By
Tashiro's theorem \cite{tashiro}, $(M,g)$ is isometric to Euclidean space.
Under this isometry, $X=x-x_0$ for some $x_0\in\mathbb R^n$. Therefore,
$$
        \mathbf v=\frac{\lambda}{n}(x-x_0)\, ,
$$
that is,
$$
        |\nabla w|^{n-2}\nabla w=\frac{\lambda}{n}(x-x_0)\, .
$$
Hence $w$ is radial around $x_0$. Writing $r=|x-x_0|$, we get
$$
        |w'|^{n-2}w'=\frac{\lambda}{n}r.
$$
Since $\lambda>0$, this gives
$$
        w'(r)=\left(\frac{\lambda}{n}\right)^{1/(n-1)}
        r^{1/(n-1)}\, .
$$
Therefore
$$
        w(x)=C_1+C_2|x-x_0|^{\frac n{n-1}},
        \qquad C_2>0\, .
$$
Since $u=-n\log w$, this gives
$$
        u=\mathcal U_{\lambda,x_0}\, ,
$$
for a suitable $\lambda>0$. The proof is complete.
\end{proof}

\

We now show that the coefficient
$$
\delta_0:=\frac{n^2}{n-1}
$$
in Theorem~\ref{teo1} is sharp. More precisely, for every $\delta>\delta_0$ we construct a complete nonflat rotationally symmetric manifold with nonnegative Ricci curvature carrying a radial solution of
$$
-\Delta_n u=e^u
$$
whose logarithmic decay is between $-\delta_0\log r$ and $-\delta\log r$. Therefore the coefficient $\delta_0$ cannot be replaced by any larger one.

We first record the elementary warped product construction which will be used in the example.

\begin{lemma}
\label{lem_wpm}
Let $a\in(0,1)$. There exists a smooth function $\psi:[0,+\infty)\to[0,+\infty)$
such that
$$
\psi(s)=s\quad\text{for }s\le1,
$$
$$
\psi(s)=as+c\quad\text{for }s\ge2
$$
for some constant $c>0$, and
$$
a\leq \psi'(s)\leq1,\qquad \psi''(s)\leq0
$$
for every $s\ge0$. Consequently, the rotationally symmetric metric
$$
g=ds^2+\psi(s)^2g_{\mathbb S^{n-1}}
$$
is a smooth complete nonflat metric on $\mathbb R^n$ with
$ \operatorname{Ric}_g\ge0$.
\end{lemma}

\begin{proof}
Choose a smooth nonincreasing function $\chi:[0,+\infty)\to[a,1]$
such that
$$
\chi\equiv1\quad\text{on }[0,1], \qquad \chi\equiv a\quad\text{on }[2,+\infty),
$$
and $\chi'\le0$. Then it is immediate to see that the function
$$
\psi(s):=\int_0^s\chi(t)\,dt
$$
satisfies the required properties. For the warped product metric (see e.g. \cite{pet})
$$
g=ds^2+\psi(s)^2g_{\mathbb S^{n-1}},
$$
the Ricci tensor in the radial direction is given by
$$
\operatorname{Ric}_{ss}=-(n-1)\frac{\psi''}{\psi}
$$
and, in the tangential directions,
$$
\operatorname{Ric}_{\theta\theta}=\frac{(n-2)(1-(\psi')^2)-\psi\psi''}{\psi^2}\,g_{\theta\theta}.
$$
Since $\psi''\le0$ and $0\le\psi'\le1$, both quantities are nonnegative.
Hence $\operatorname{Ric}_g\ge0$. The metric is complete because $s$ is the
distance from the pole and ranges in $[0,+\infty)$. It is nonflat since
$a<1$, so $\psi'$ is not identically equal to $1$.
\end{proof}

We next solve the radial equation on this model.

\begin{lemma}
\label{lem_wpm_sol}
Let $g=ds^2+\psi(s)^2g_{\mathbb S^{n-1}}$ be as in Lemma \ref{lem_wpm}.
For every \(\kappa\in\mathbb R\) there exists a global radial weak solution
$u_\kappa=u_\kappa(s)$ of
$$
-\Delta_n u_\kappa=e^{u_\kappa}
$$
in $(\mathbb R^n,g)$, regular at the pole, such that
$$
u_\kappa(0)=\kappa, \qquad u_\kappa'(s)<0\quad\text{for }s>0.
$$
Moreover, setting
$$
A_\kappa(s):=\psi(s)^{n-1}(-u_\kappa'(s))^{n-1},
$$
the limit
$$
L_\kappa:=\lim_{s\to+\infty}A_\kappa(s) = \int_0^{+\infty}\psi(s)^{n-1}e^{u_\kappa(s)}\,ds
$$
exists, it is finite and satisfies
$$
L_\kappa^{\frac1{n-1}}\le \delta_0:=\frac{n^2}{n-1},
$$
and
$$
\lim_{\kappa\to+\infty}L_\kappa^{\frac1{n-1}}=\delta_0.
$$
Finally,
$$
\lim_{s\to+\infty}\frac{u_\kappa(s)}{\log s}=-\frac{L_\kappa^{\frac1{n-1}}}{a}.
$$
\end{lemma}

\begin{proof}
For a radial function $u=u(s)$ on $(\mathbb{R}^n,g)$, one has
$$
\Delta_n u=\frac{1}{\psi^{n-1}}\left(\psi^{n-1}|u'|^{n-2}u'\right)'.
$$
We look for decreasing solutions and set
$$
y:=-u'.
$$
Then $|u'|^{n-2}u'=-y^{n-1}$ and the equation
$$
-\Delta_n u=e^u
$$
is equivalent to
$$
\left(\psi^{n-1}y^{n-1}\right)'=\psi^{n-1}e^u.
$$
Equivalently, with
$$
A(s):=\psi(s)^{n-1}y(s)^{n-1},
$$
we obtain the first-order system
\begin{equation}\label{4}
\begin{cases}
A'(s)=\psi(s)^{n-1}e^{u(s)},\\[4pt]
u'(s)=-A(s)^{\frac1{n-1}}\psi(s)^{-1}.
\end{cases}
\end{equation}
The regularity condition at the pole is
$$
A(0)=0, \qquad u(0)=\kappa.
$$
Since $\psi(s)=s$ near $s=0$, the system \eqref{4} admits near the pole the local solution given by the regular radial Euclidean solution with initial value $\kappa$.
Let us consider such solution $(A_\kappa,u_\kappa)$ of \eqref{4}, defined on its maximal interval of existence.

We \emph{claim} that it is defined for every $s\geq0$. Indeed, we start noting that $A_\kappa'>0$, so that $A_\kappa$ is strictly monotone increasing and hence strictly positive for $s>0$. Then $u_\kappa'<0$, so that $u_\kappa$ is strictly monotone decreasing.

We now prove a uniform bound on $A_\kappa$. By \eqref{4} we compute
$$
\left(A_\kappa^{\frac n{n-1}}\right)'=\frac n{n-1}A_\kappa^{\frac1{n-1}}A_\kappa'=\frac n{n-1}(-u_\kappa'\psi)\psi^{n-1}e^{u_\kappa}=-\frac n{n-1}\psi^n(e^{u_\kappa})'\,.
$$
Integrating the previous identity from $0$ to $R$ we get
$$
A_\kappa(R)^{\frac n{n-1}}=-\frac n{n-1}\psi(R)^ne^{u_\kappa(R)}+\frac{n^2}{n-1}\int_0^R\psi'\psi^{n-1}e^{u_\kappa}\,ds\,,
$$
where we integrated by parts and we used $A_\kappa(0)=0$ and $\psi(0)=0$.
Since $0\le\psi'\le1$ we deduce
$$
A_\kappa(R)^{\frac n{n-1}}\le\frac{n^2}{n-1}\int_0^R\psi^{n-1}e^{u_\kappa}\,ds.
$$
From \eqref{4} and $A_\kappa(0)=0$ we have
\begin{equation}\label{15}
\int_0^R\psi^{n-1}e^{u_\kappa}\,ds=A_\kappa(R)\,,
\end{equation}
hence
$$
A_\kappa(R)^{\frac n{n-1}}\leq \frac{n^2}{n-1}A_\kappa(R)=\delta_0 A_\kappa(R)\qquad\textrm{for every }R>0.
$$
Since $A_\kappa(R)>0$ for $R>0$, we get
$$
A_\kappa(R)^{\frac1{n-1}}\le\delta_0\qquad\text{for every }R>0\,.
$$
Thus $A_\kappa$ is strictly increasing and bounded on its maximal interval of existence; from \eqref{4} we see that $u_\kappa'$ is bounded on bounded intervals, and hence also $u_\kappa$. From standard ODE theory, we have that $(A_\kappa,u_\kappa)$ must be defined for every $s\geq0$. We have thus proven our \emph{claim}.

Since $A_\kappa$ is increasing and bounded, also using \eqref{15}, we have
\begin{equation}\label{18}
L_\kappa:=\lim_{R\to+\infty}A_\kappa(R)=\int_0^{+\infty}\psi(s)^{n-1}e^{u_\kappa(s)}\,ds
\end{equation}
exists and satisfies
\begin{equation}\label{17}
L_\kappa^{\frac1{n-1}}\le\delta_0\,.
\end{equation}
We now show that
$$
\lim_{\kappa\to+\infty}L_\kappa^{\frac1{n-1}}=\delta_0.
$$
The upper bound has already been proved. We prove the lower bound. Let
$$
\varepsilon_\kappa:=e^{-\frac{\kappa}{n}}.
$$
Since $\psi(s)=s$ for $s\in[0,1]$, for every $\kappa$ sufficiently large the function
$$
U_\kappa(\rho):=u_\kappa(\varepsilon_\kappa\rho)-\kappa
$$
solves
$$
-\Delta_n U_\kappa=e^{U_\kappa}
$$
as a radial equation in $B_{\varepsilon_\kappa^{-1}}(0)\subset(\mathbb{R}^n,g_\textrm{Eucl})$, with
$$
U_\kappa(0)=0,\qquad U_\kappa'(0)=0.
$$
By uniqueness of the regular radial Euclidean solution, $U_\kappa$ coincides with the normalized Euclidean bubble
$$
U_*(\rho)=-n\log\left(1+c_n^{-\frac n{n-1}}\rho^{\frac n{n-1}}\right),
$$
in $B_{\varepsilon_\kappa^{-1}}(0)$, where
$$
c_n=n^{1/n}\left(\frac{n^2}{n-1}\right)^{\frac{n-1}{n}}.
$$
In particular,
\begin{equation}\label{16}
\rho^{n-1}\left(-U_*'(\rho)\right)^{n-1}\longrightarrow\delta_0^{\,n-1}\qquad\text{as }\rho\to+\infty.
\end{equation}
For every $\rho\in(0,\varepsilon_\kappa^{-1})$ and $\kappa$ large,
$$
\psi(\varepsilon_\kappa\rho)=\varepsilon_\kappa\rho
$$
and
$$
u_\kappa'(\varepsilon_\kappa\rho)=\varepsilon_\kappa^{-1}U_*'(\rho)\,.
$$
Therefore
$$
A_\kappa(\varepsilon_\kappa\rho)=\psi(\varepsilon_\kappa\rho)^{n-1}\left(-u_\kappa'(\varepsilon_\kappa\rho)\right)^{n-1}=\rho^{n-1}\left(-U_*'(\rho)\right)^{n-1}.
$$
Since $A_\kappa$ is increasing and converges to $L_\kappa$, we have for every $\kappa$ large and every $\rho\in(0,\varepsilon_\kappa^{-1})$
$$
L_\kappa\ge A_\kappa(\varepsilon_\kappa\rho)=\rho^{n-1}\left(-U_*'(\rho)\right)^{n-1}.
$$
By continuity we can choose $\rho=\varepsilon_\kappa^{-1}$ in the previous inequality, and then pass to the $\liminf_{\kappa\to+\infty}$. Therefore, also using \eqref{16}, we obtain
$$
\liminf_{\kappa\to+\infty}L_\kappa\geq\delta_0^{\,n-1}.
$$
Together with the upper bound \eqref{17} this proves
$$
\lim_{\kappa\rightarrow+\infty}L_\kappa^{\frac1{n-1}}=\delta_0.
$$
Finally, since $\psi(s)=as+c$ for $s\ge2$, by \eqref{18} and \eqref{4}, we have
$$
su_\kappa'(s)=-\frac{sA_\kappa(s)^{\frac1{n-1}}}{\psi(s)}\longrightarrow-\frac{L_\kappa^{\frac1{n-1}}}{a}
$$
as $s$ tends to $+\infty$. It follows that $u_\kappa$ diverges to $-\infty$ at infinity and
$$
\lim_{s\to+\infty}\frac{u_\kappa(s)}{\log s}=-\frac{L_\kappa^{\frac1{n-1}}}{a}.
$$
This completes the proof.
\end{proof}

We can now complete the sharpness argument.

\begin{example}[Sharpness of the leading coefficient]\label{ex_sharpness}
Let
$$
\delta>\delta_0.
$$
Choose
$$
a\in\left(\frac{\delta_0}{\delta},1\right).
$$
Then
$$
\delta_0<\frac{\delta_0}{a}<\delta.
$$
Let $g=ds^2+\psi(s)^2g_{\mathbb S^{n-1}}$ be the corresponding warped product
metric given by Lemma \ref{lem_wpm}. By Lemma~\ref{lem_wpm_sol},
$$
\frac{L_\kappa^{\frac1{n-1}}}{a}\longrightarrow\frac{\delta_0}{a}\qquad\text{as }\kappa\to+\infty.
$$
Hence, for $\kappa$ large enough,
$$
\delta_0<\delta_\kappa:=\frac{L_\kappa^{\frac1{n-1}}}{a}<\delta.
$$
For this value of $\kappa$, the corresponding solution satisfies
$$
\lim_{s\to+\infty}\frac{u_\kappa(s)}{\log s}=-\delta_\kappa\,
$$
i.e. $u_\kappa(s)\sim-\delta_\kappa\log s$. Since $\delta_\kappa\in(\delta_0,\delta)$, we have, for all sufficiently large $s$,
$$
-\delta_0\log s\ge u_\kappa(s)\geq -\delta\log s.
$$
Since $s$ is the geodesic distance from the pole, this gives
$$
-\delta_0\log r(x)\ge u_\kappa(x)\ge -\delta\log r(x)
$$
outside a compact set.

The manifold $(\mathbb R^n,g)$ is complete, nonflat, and has nonnegative Ricci
curvature, while \(u_\kappa\) solves
$$
-\Delta_n u_\kappa=e^{u_\kappa}.
$$
Moreover, since $\psi(s)=as+c$ for $s\ge2$, we have
$$
\operatorname{Vol}(B_R)=\sigma_{n-1}\int_0^R\psi(s)^{n-1}\,ds\sim\sigma_{n-1}\frac{a^{n-1}}{n}R^n=\omega_n a^{n-1}R^n .
$$
Hence
$$
\operatorname{AVR}(\mathbb R^n,g)=a^{n-1}>0.
$$
Furthermore, by \eqref{18},
$$
\int_{\mathbb R^n}e^{u_\kappa}\,dV_g=\sigma_{n-1}\int_0^{+\infty}\psi(s)^{n-1}e^{u_\kappa(s)}\,ds=\sigma_{n-1}L_\kappa<+\infty.
$$
and hence $u$ has finite mass.

Therefore any rigidity statement with the leading coefficient
$\delta_0=\frac{n^2}{n-1}$ replaced by a larger coefficient $\delta>\delta_0$
would be false. This proves the sharpness of the coefficient $-\delta_0$ in
Theorem \ref{teo1}.
\end{example}

\

\section{Proof of Proposition \ref{prop_log_lb} and Theorem \ref{teo2}}
\label{sec:finite-mass}

In this section we prove a rigidity result under a finite-mass assumption. The
argument is different from the $P-$function method used in the proof of
Theorem~\ref{teo1}. It is closer in spirit to the two-dimensional approach
of Cai--Lai \cite{CaiLai}, but the logarithmic potential estimate is replaced by a nonlinear capacitary estimate for the $n-$Laplacian.

We split the proof into two independent ingredients. The first one is a sharp
isoperimetric consequence of the equation and the finite-mass assumption.

\begin{lemma}\label{lem_mlb}
Let $(M^n,g)$ be a complete noncompact Riemannian manifold with
$\operatorname{Ric}\ge0$ and positive asymptotic volume ratio
$\theta=\operatorname{AVR}(M,g)>0$. Let $u$ be a weak solution of
$$
-\Delta_n u=e^u
$$
such that
$$
\mathcal{M}:=\int_M e^u\,dV_g<+\infty.
$$
Then
$$
\mathcal{M}^{\frac1{n-1}}\geq\frac{n^2}{n-1}\left(\theta\sigma_{n-1}\right)^{\frac1{n-1}}.
$$
\end{lemma}

\begin{proof}
For $t\in\mathbb R$, set
$$
\Omega_t:=\{x\in M:u(x)>t\},\quad \text{ and } \quad  F(t):=\int_{\Omega_t}e^u\,dV_g.
$$
Since $e^u\in L^1(M)$, we have
$$
\operatorname{Vol}(\Omega_t)=\int_{\Omega_t}e^u e^{-u}\,dV_g\leq e^{-t}\mathcal{M}<+\infty
$$
for every $t\in\mathbb R$.

For almost every regular value $t$, using the equation and the divergence
theorem on $\Omega_t$, we get
\begin{equation}\label{19}
F(t)=\int_{\Omega_t}e^u\,dV_g=-\int_{\Omega_t}\Delta_n u\,dV_g=\int_{\partial\Omega_t}|\nabla u|^{n-1}\,d\sigma_g.
\end{equation}
Moreover, by the coarea formula,
\begin{equation}\label{20}
-F'(t)=e^t\int_{\partial\Omega_t}\frac{1}{|\nabla u|}\,d\sigma_g
\end{equation}
for almost every $t$. By Holder's inequality,
$$
|\partial\Omega_t|\leq\left(\int_{\partial\Omega_t}|\nabla u|^{n-1}\,d\sigma_g\right)^{1/n}\left(\int_{\partial\Omega_t}\frac{1}{|\nabla u|}\,d\sigma_g\right)^{(n-1)/n}.
$$
Thus, using \eqref{19},
\begin{equation}\label{21}
\int_{\partial\Omega_t}\frac{1}{|\nabla u|}\,d\sigma_g\geq|\partial\Omega_t|^{\frac n{n-1}}F(t)^{-\frac1{n-1}}.
\end{equation}
Combining \eqref{20} and \eqref{21}, we obtain
$$
-F'(t)\geq e^t|\partial\Omega_t|^{\frac n{n-1}}F(t)^{-\frac1{n-1}}.
$$
Equivalently,
\begin{equation}\label{22}
-\left(F(t)^{\frac n{n-1}}\right)'\geq \frac n{n-1}e^t|\partial\Omega_t|^{\frac n{n-1}}.
\end{equation}
We now use the sharp isoperimetric inequality on complete manifolds with
nonnegative Ricci curvature and asymptotic volume ratio $\theta$:
\begin{equation}\label{23}
|\partial\Omega|^{\frac n{n-1}}\geq n\left(\theta\sigma_{n-1}\right)^{\frac1{n-1}}\operatorname{Vol}(\Omega)
\end{equation}
for every bounded finite-perimeter set $\Omega$. Applying this to $\Omega_t$
in \eqref{22}, we get
\begin{equation}\label{24}
-\left(F(t)^{\frac n{n-1}}\right)'\geq \frac{n^2}{n-1} \left(\theta\sigma_{n-1}\right)^{\frac1{n-1}} e^t\operatorname{Vol}(\Omega_t).
\end{equation}
Integrating \eqref{24} from $-\infty$ to $+\infty$, and using
$$
\lim_{t\to-\infty}F(t)=\mathcal{M},\qquad\lim_{t\to+\infty}F(t)=0,
$$
we obtain
$$
\mathcal{M}^{\frac n{n-1}}\geq\frac{n^2}{n-1}\left(\theta\sigma_{n-1}\right)^{\frac1{n-1}}\int_{-\infty}^{+\infty}e^t\operatorname{Vol}(\Omega_t)\,dt.
$$
By the layer-cake formula,
$$
\int_{-\infty}^{+\infty}e^t\operatorname{Vol}(\Omega_t)\,dt=\int_M e^u\,dV_g=\mathcal{M}.
$$
Therefore
$$
\mathcal{M}^{\frac n{n-1}}\geq\frac{n^2}{n-1}\left(\theta\sigma_{n-1}\right)^{\frac1{n-1}}\mathcal{M}.
$$
Since $\mathcal{M}>0$, this completes the proof.
\end{proof}

The second ingredient is the nonlinear analogue of the logarithmic potential
lower estimate. It is a capacitary consequence of the fact that
$-u$ has finite positive $n-$Riesz measure. We start with the following general Lemma.

\begin{lemma}\label{lem_cap}
Let $\Omega\Subset M$ be a bounded open domain and let
$K\Subset\Omega$ be compact. Let
$$
z\in W^{1,n}(\Omega)\cap L^\infty(\Omega)
$$
satisfy
\begin{equation}\label{26}
  \Delta_n z=f
\end{equation}
in the weak sense in $\Omega$, with $f\in L^1(\Omega)$ and $f\geq0$ a.e. in $\Omega$. Then
\begin{equation}\label{43}
\mu(K):=\int_{K}f\,dV_g\leq\operatorname{Cap}_n(K,\Omega)
\left(\sup_{\Omega} z- \inf_{K} z\right)^{n-1}\,,
\end{equation}
where
$$
\operatorname{Cap}_n(K,\Omega):=\inf\left\{\int_\Omega |\nabla\varphi|^n\,dV_g:\varphi\in W^{1,n}_0(\Omega),\,0\le \varphi\le 1,\,\varphi=1 \text{ in a neighbourhood of }K\right\}.
$$
\end{lemma}

\begin{remark}\label{rem1}
  We note that in particular, under the assumptions of Lemma \ref{lem_cap}, since $\Delta_n z\ge0$ then by the weak maximum principle for
$n$-subharmonic functions we have
$$
\mu(K)\leq\operatorname{Cap}_n(K,\Omega)\left(\sup_{\partial\Omega} z-\inf_{K} z\right)^{n-1},
$$
whenever the boundary essential supremum is well defined, for instance if
$z$ admits a continuous representative up to $\partial\Omega$.
\end{remark}

\begin{proof}
Set
$$
L:=\sup_{\Omega} z, \qquad m:=\inf_{K} z .
$$
Fix $\varepsilon>0$, and define
$$
v:=L-z+\varepsilon .
$$
Then $v\ge\varepsilon>0$ a.e. in $\Omega$ and
\begin{equation}\label{14}
-\Delta_n v=f
\end{equation}
in the weak sense in $\Omega$. Let
$\varphi\in W^{1,n}_0(\Omega)$, $0\leq\varphi\leq1$ a.e in $\Omega$, $\varphi=1$ in a neighbourhood of $K$ and define
$$
\eta:=\varphi^n v^{1-n}.
$$
Then $\eta\geq0$ a.e in $\Omega$, $\eta\in W^{1,n}_0(\Omega)\cap L^\infty(\Omega)$ and by \eqref{14}
$$
\int_\Omega \eta f\,dV_g = \int_\Omega |\nabla v|^{n-2}\langle\nabla v,\nabla\eta\rangle\,dV_g.
$$
Thus we obtain
$$
\begin{aligned}
\int_\Omega \varphi^n v^{1-n}f\,dV_g &=n\int_\Omega\varphi^{n-1}v^{1-n}|\nabla v|^{n-2}\langle\nabla v,\nabla\varphi\rangle\,dV_g-(n-1)\int_\Omega
\varphi^n v^{-n}|\nabla v|^n\,dV_g\\
&\leq n\int_\Omega\varphi^{n-1}v^{1-n}|\nabla v|^{n-1}|\nabla\varphi|\,dV_g-(n-1)\int_\Omega\varphi^n v^{-n}|\nabla v|^n\,dV_g.
\end{aligned}
$$
by the Cauchy-Schwarz inequality. We now use the elementary inequality
$$
n a^{n-1}b-(n-1)a^n\le b^n \qquad\text{for all }a,b\ge0
$$
with
$$
a:=\frac{\varphi|\nabla v|}{v},\qquad b:=|\nabla\varphi|,
$$
and thus we obtain
$$
n\varphi^{n-1}v^{1-n}|\nabla v|^{n-1}|\nabla\varphi|-(n-1)\varphi^n v^{-n}|\nabla v|^n\leq|\nabla\varphi|^n
$$
a.e. in $\Omega$. Therefore
\begin{equation}\label{27}
\int_\Omega \varphi^n v^{1-n}f\,dV_g\leq\int_\Omega |\nabla\varphi|^n\,dV_g\, . 
\end{equation}
On the other hand we have
$$
v^{1-n}\ge (L-m+\varepsilon)^{1-n}
$$
a.e. on $K$.
Moreover, $\varphi=1$ in a neighbourhood of $K$, hence
\begin{equation}\label{28}
\int_\Omega \varphi^n v^{1-n}f\,dV_g\geq\int_K \varphi^n v^{1-n}f\,dV_g\geq(L-m+\varepsilon)^{1-n}\mu(K).
\end{equation}
Combining \eqref{27} and \eqref{28} we obtain
$$
\mu(K)\leq(L-m+\varepsilon)^{n-1}\int_\Omega |\nabla\varphi|^n\,dV_g.
$$
Taking the infimum over all admissible $\varphi$, we get
$$
\mu(K)
\le
(L-m+\varepsilon)^{n-1}
\operatorname{Cap}_n(K,\Omega).
$$
Letting $\varepsilon$ tend to $0$, we conclude that
$$
\mu(K)\leq\operatorname{Cap}_n(K,\Omega)\left(\sup_{\Omega} z-\inf_{K} z\right)^{n-1}.
$$
and the proof is complete.
\end{proof}

\begin{lemma}\label{lem_cle}
Let $(M^n,g)$ be a complete noncompact Riemannian manifold with $\operatorname{Ric}\ge0$ and positive asymptotic volume ratio $\theta=\operatorname{AVR}(M,g)>0$. Let $u$ be a weak solution of
$$
-\Delta_n u=e^u
$$
such that
$$
\mathcal{M}:=\int_Me^u\,dV_g<+\infty.
$$
Then, for every fixed $o\in M$,
\begin{equation}\label{25}
\liminf_{R\to+\infty}\frac{-\inf_{\partial B_R(o)}u}{\log R}\geq\left(\frac{\mathcal{M}}{\theta\sigma_{n-1}}\right)^{\frac1{n-1}}.
\end{equation}
\end{lemma}

\begin{proof}
Set $z:=-u$, then
$$
\Delta_n z=f:=e^u\geq0
$$
in the weak sense and $z$ is continuous. For any measurable set $E\subset M$ let $\mu(E):=\int_E f\,dV_g$, then we have $\mu(M)=\mathcal{M}<+\infty$. Fix $R>3\rho>0$, applying Lemma \ref{lem_cap} and Remark \ref{rem1} with $K=B_\rho(o)$ and $\Omega=B_R(o)$ we deduce
\begin{equation}\label{313}
-\inf_{\partial B_R(o)}u=\sup_{\partial B_R(o)}z\geq\inf_{B_\rho(o)}z+\left(\frac{\mu(B_\rho(o))}{\operatorname{Cap}_n(B_\rho(o),B_R(o))}\right)^{\frac1{n-1}}.
\end{equation}
We now estimate the capacity from above. Consider the Lipschitz radial cut-off
$$
\varphi_R(x)=\frac{\log R-\log r(x)}{\log R-\log\rho} \qquad\textrm{on }B_R(o)\setminus B_\rho(o)\,,
$$
extended by $1$ on $B_\rho(o)$ and by $0$ outside $B_R(o)$. Then by a standard approximation argument $\varphi_R$ is an admissible test function and
$$
|\nabla\varphi_R|=\frac{1}{r(x)\log(R/\rho)}\qquad\textrm{a.e. in }B_R(o)\setminus B_\rho(o)\,.$$
Hence
\begin{equation}\label{32}
\operatorname{Cap}_n(B_\rho(o),B_R(o))\leq\frac{1}{\log(R/\rho)^n}\int_{B_R(o)\setminus B_\rho(o)}\frac{1}{r(x)^n}\,dV_g\,.
\end{equation}
By the coarea formula,
$$
\int_{B_R(o)\setminus B_\rho(o)}\frac{1}{r(x)^n}\,dV_g=\int_\rho^R\frac{A(t)}{t^n}\,dt,
$$
where $A(t)=\mathcal H^{n-1}(\partial B_t(o))$ for a.e. $t$. If $V(t):=\operatorname{Vol}(B_t(o))$ then $A(t)=V'(t)$ for a.e. $t$ and, integrating by parts, we obtain
\begin{equation}\label{34}
\int_{B_R(o)\setminus B_\rho(o)}\frac{1}{r(x)^n}\,dV_g=\int_\rho^R\frac{V'(t)}{t^n}\,dt=\frac{V(R)}{R^n}-\frac{V(\rho)}{\rho^n}+n\int_\rho^R\frac{V(t)}{t^{n+1}}\,dt.
\end{equation}
Since
$$
\frac{V(t)}{\omega_n t^n}\rightarrow\theta^+ \qquad\text{as }t\rightarrow+\infty
$$
for every $\varepsilon>0$ there exists $\rho_0$ such that for every $t>\rho_0$
$$
\theta\omega_n t^n\leq V(t)\leq(\theta+\tfrac{\varepsilon}{2})\omega_n t^n\,.
$$
Let $R>3\rho$ with $\rho>\rho_0$, then
\begin{equation}\label{33}
\int_\rho^R\frac{V'(t)}{t^n}\,dt\leq\tfrac{\varepsilon}{2}\omega_n+ n(\theta+\tfrac{\varepsilon}{2})\omega_n\log\frac{R}{\rho}\leq n(\theta+\varepsilon)\omega_n\log\frac{R}{\rho}\,.
\end{equation}
Combining \eqref{32}, \eqref{34} and \eqref{33}, we get
\begin{equation}\label{36}
\operatorname{Cap}_n(B_\rho(o),B_R(o))\leq\frac{n(\theta+\varepsilon)\omega_{n}}{\left(\log \frac{R}{\rho}\right)^{n-1}}=\frac{(\theta+\varepsilon)\sigma_{n-1}}{\left(\log \frac{R}{\rho}\right)^{n-1}}\,.
\end{equation}
Inserting \eqref{36} into \eqref{313}, dividing by $\log R$ and letting
$R$ go to $\infty$, we obtain for every $\varepsilon>0$, $\rho>\rho_0$
$$
\liminf_{R\to+\infty}\frac{-\inf_{\partial B_R(o)}u}{\log R}\geq\left(\frac{\mu(B_\rho(o))}{(\theta+\varepsilon)\sigma_{n-1}}\right)^{\frac1{n-1}}\,.
$$
Finally, letting $\rho$ go to $\infty$
$$
\liminf_{R\to+\infty}\frac{-\inf_{\partial B_R(o)}u}{\log R}\geq\left(\frac{\mathcal{M}}{(\theta+\varepsilon)\sigma_{n-1}}\right)^{\frac1{n-1}}\,.
$$
for every $\varepsilon>0$. Given the arbitrariness of $\varepsilon$, we get \eqref{25}.
\end{proof}

We present now the proof of Proposition \ref{prop_log_lb}.

\begin{proof}[Proof of Proposition \ref{prop_log_lb}]
Set $z:=-u$, let $f:=e^u$ and for any measurable set $E\subset M$ let $\mu(E):=\int_Ef\,dV_g$. Note that $u$, and thus also $z$, is continuous. Since $e^u>0$, we can choose a smooth relatively compact domain $K\Subset M$ such that
$$
\mu(K)>0.
$$
Fix $\rho>R_0$ such that $K\subset B_\rho(o)$ and let $R>2\rho$. Applying Lemma \ref{lem_cap} and Remark \ref{rem1} to $z$ with
$\Omega=B_R(o)$, we get
\begin{equation}\label{45}
\mu(K) \leq\operatorname{Cap}_n(K,B_R(o))\left(\sup_{\partial B_R(o)} z-\inf_{K}z\right)^{n-1}.
\end{equation}
By the logarithmic lower bound \eqref{41}, for $R$ large enough,
$$
z(x)=-u(x)\le \beta\log R\qquad\text{on }\partial B_R(o)\,.
$$
Therefore
$$
\left(\sup_{\partial B_R(o)} z-\inf_{K}z\right)\leq C_K+\beta\log R\leq C\log R
$$
for all large $R$. Since $\mu(K)>0$, \eqref{45} gives
\begin{equation}\label{46}
\operatorname{Cap}_n(K,B_R(o))\geq\frac{C}{(\log R)^{n-1}}
\end{equation}
for some constant $C>0$ independent of $R$. We now estimate the same capacity from above. Since $K\subset B_\rho(o)$, the
radial logarithmic cut-off
$$
\varphi_R(x)=\frac{\log R-\log r(x)}{\log R-\log \rho}
$$
on $B_R(o)\setminus B_\rho(o)$, extended by $1$ on $B_\rho(o)$ and by $0$
outside $B_R(o)$, is an admissible competitor. Hence
\begin{equation}\label{47}
\operatorname{Cap}_n(K,B_R(o))\leq\frac{1}{\log(R/\rho)^n}\int_{B_R(o)\setminus B_\rho(o)}\frac{1}{r(x)^n}\,dV_g.
\end{equation}
Combining \eqref{46} and \eqref{47}, we obtain
\begin{equation}\label{48}
\int_{B_R(o)\setminus B_\rho(o)}\frac{1}{r(x)^n}\,dV_g\geq C\frac{(\log R-\log\rho)^{n}}{(\log R)^{n-1}}\, , 
\end{equation}
for all large $R$ and $\rho$. Let $V(R):=\operatorname{Vol}(B_R(o))$. Using the coarea formula and integrating by parts we obtain
\begin{equation}\label{49}
\int_{B_R(o)\setminus B_\rho(o)}\frac{1}{r(x)^n}\,dV_g= \int_\rho^R\frac{V'(t)}{t^n}\,dt=\frac{V(R)}{R^n}-\frac{V(\rho)}{\rho^n}+n\int_\rho^R\frac{V(t)}{t^{n+1}}\,dt\,.
\end{equation}
Now assume by contradiction that
$$
\operatorname{AVR}(M,g)=0.
$$
By Bishop--Gromov monotonicity this is equivalent to
\begin{equation}\label{50}
\frac{V(t)}{t^n}\longrightarrow0\qquad\text{as }t\to+\infty.
\end{equation}
For every $\varepsilon>0$ there exists $R^*>R_0$ such that
$$
0<\frac{V(t)}{t^n}<\varepsilon\qquad\text{for every }t>R^*.
$$
Then if $R>2\rho>2R^*$ we obtain
$$
0<\frac{1}{\log R}\int_\rho^R\frac{V(t)}{t^{n+1}}\,dt\leq\frac{\varepsilon}{\log R}\int_\rho^R\frac{1}{t}\,dt\leq\varepsilon\,.
$$
We deduce that if $\rho>R^*$
\begin{equation}\label{51}
\int_\rho^R\frac{V(t)}{t^{n+1}}\,dt=o(\log R)\qquad\text{as }R\rightarrow\infty\,.
\end{equation}
Therefore \eqref{49}, \eqref{50} and \eqref{51} give
$$
\int_{B_R(o)\setminus B_\rho(o)}\frac{1}{r(x)^n}\,dV_g=o(\log R),
$$
which contradicts \eqref{48}. Hence
$$
\operatorname{AVR}(M,g)>0\,.
$$
It remains to prove the one-endedness conclusion. If $M$ had at least two ends,
then by the Cheeger--Gromoll splitting theorem (see e.g. \cite[Section 4.3]{pet}),
$$
(M,g)\simeq (N^{n-1}\times\mathbb R,\,g_N+dt^2),
$$
with $N$ compact. Hence
$$
\operatorname{Vol}(B_R(o))=O(R),
$$
and therefore
$$
\operatorname{AVR}(M,g)=0,
$$
contradicting what we have just proved. Thus $M$ has only one end, and the proof is complete.
\end{proof}

Finally, we give the proof of Theorem \ref{teo2}.

\begin{proof}[Proof of Theorem \ref{teo2}]
We start noting that by Proposition \ref{prop_log_lb} we have $\operatorname{AVR}(M,g)>0$. By Lemma \ref{lem_cle},
$$
\liminf_{R\to+\infty}\frac{-\inf_{\partial B_R(o)}u}{\log R}\geq\left(\frac{\mathcal{M}}{\theta\sigma_{n-1}}\right)^{\frac1{n-1}}.
$$
On the other hand, the asymptotic lower bound \eqref{37} gives
$$
-\inf_{\partial B_R(o)}u\leq\frac{n^2}{n-1}\log R+o(\log R).
$$
Therefore,
$$
\left(\frac{\mathcal{M}}{\theta\sigma_{n-1}}\right)^{\frac1{n-1}}\leq \frac{n^2}{n-1}.
$$
Equivalently,
\begin{equation}\label{38}
\mathcal{M}^{\frac1{n-1}}\leq\frac{n^2}{n-1}\left(\theta\sigma_{n-1}\right)^{\frac1{n-1}}.
\end{equation}
By Lemma \ref{lem_mlb}, the opposite inequality holds:
\begin{equation}\label{39}
\mathcal{M}^{\frac1{n-1}}\geq\frac{n^2}{n-1}\left(\theta\sigma_{n-1}\right)^{\frac1{n-1}}.
\end{equation}
Hence equality holds in \eqref{38}--\eqref{39}.

Inspecting the proof of Lemma \ref{lem_mlb}, equality in the final
mass inequality implies equality in the sharp isoperimetric inequality
$$
|\partial\Omega_t|^{\frac n{n-1}}=n\left(\theta\sigma_{n-1}\right)^{\frac1{n-1}}\operatorname{Vol}(\Omega_t)
$$
with $\Omega_t=\{x\in M:u(x)>t\}$ for a.e. regular value $t$ such that $0<\operatorname{Vol}(\Omega_t)<+\infty$.
By the equality case in the sharp isoperimetric inequality on complete manifolds
with nonnegative Ricci curvature and positive asymptotic volume ratio, this
forces $(M,g)$ to be isometric to the Euclidean space, see \cite{Bre} and \cite{BaKr}. Then $u$ is a solution of \eqref{eq_crit} in $\mathbb{R}^n$ satisfying \eqref{40}. By the classification result in \cite{Esp} we have $u=\mathcal{U}_{\lambda,x_0}$ for some $x_0\in\mathbb{R}^n$, $\lambda>0$, with $\mathcal{U}_{\lambda,x_0}$ as in \eqref{Tal_n}. This completes the proof.
\end{proof}

We present here some examples showing the sharpness of the assumptions in Proposition \ref{prop_log_lb} and Theorem \ref{teo2}.

\begin{example}\label{exe2}
Let $(N^{n-1},h)$ be a compact Riemannian manifold with
$$
\operatorname{Ric}_h\ge0,
$$
and consider the product manifold
$$
(M^n,g)=(N^{n-1}\times\mathbb R,\;h+dt^2).
$$
Then $\operatorname{Ric}_g\ge0$. We construct weak solutions of
$$
-\Delta_n u=e^u\qquad\text{in }M
$$
with linear decay at infinity. Let $u=u(t)$. Then
$$
\Delta_n u=\left(|u'|^{n-2}u'\right)'
$$
and the equation becomes
\begin{equation}\label{52}
-\left(|u'|^{n-2}u'\right)'=e^u\qquad\text{on }\mathbb R.
\end{equation}
Fix $\kappa\in\mathbb R$. We look for a solution of \eqref{52} on $[0,\infty)$ satisfying
$$
u(0)=\kappa,\qquad u'(0)=0\,.
$$
We note that $\left(|u'|^{n-2}u'\right)'<0$, hence $|u'|^{n-2}u'$ is strictly decreasing and
$$
|u'(t)|^{n-2}u'(t)<|u'(0)|^{n-2}u'(0)=0\qquad\text{for every }t>0\,.
$$
Thus $u'<0$ and $u$ is strictly decreasing for $t>0$. Equation \eqref{52} admits the first integral
\begin{equation}\label{54}
\frac{n-1}{n}|u'|^n+e^u=e^\kappa.
\end{equation}
Hence, for $t>0$,
$$
-u'(t)=\left(\frac{n}{n-1}\left(e^\kappa-e^{u(t)}\right)\right)^{1/n}\,.
$$
Equivalently, $u$ is defined implicitly by
\begin{equation}\label{53}
t=\int_{u(t)}^\kappa\frac{ds}{\left(\frac{n}{n-1}\left(e^\kappa-e^s\right)\right)^{1/n}}\qquad\text{for }t\ge0.
\end{equation}
The integral is finite near $s=\kappa$, because $e^\kappa-e^s\sim e^\kappa(\kappa-s)$ as $s$ tends to $\kappa$, and $u$ is defined for every $t\geq0$. Note that
$$
\int_{-\infty}^\kappa\frac{ds}{\left(\frac{n}{n-1}\left(e^\kappa-e^s\right)\right)^{1/n}}=+\infty,
$$
since the function under integral converges to a positive constant as $s$ tends to $-\infty$, as the denominator tends to
$$
\left(\frac{n}{n-1}e^\kappa\right)^{1/n}.
$$
Therefore \eqref{53} defines a global decreasing solution on \([0,+\infty)\), and we
extend it evenly on $\mathbb{R}$ by setting
$$
u(-t)=u(t).
$$
Then $u\in C^1(\mathbb R)$, $|u'|^{n-2}u'\in C^1(\mathbb R)$, and $u$ is a
weak solution of \eqref{52} on the whole real line. Consequently, the function
$$
u(y,t):=u(t)
$$
is a weak solution of
$$
-\Delta_n u=e^u
$$
on $N\times\mathbb R$. Moreover, since $u(t)\to-\infty$ as $t\to+\infty$, \eqref{54} gives
$$
-u'(t)\rightarrow\left(\frac{n}{n-1}e^\kappa\right)^{1/n}=:a_\kappa>0.
$$
Hence
$$
u(t)\sim-a_\kappa |t|\qquad\text{as }|t|\to+\infty.
$$
Since $N$ is compact, the geodesic distance from a fixed point of \(N\times\mathbb R\) satisfies
$$
r(y,t)\sim |t|\qquad\text{as }|t|\to+\infty.
$$
Therefore
$$
u(y,t)\sim-a_\kappa r(y,t)\qquad\text{as }r(y,t)\to+\infty.
$$
In particular, the decay is linear, hence strictly faster than logarithmic.

Finally,
$$
\int_{N\times\mathbb R}e^u\,dV_g=\operatorname{Vol}(N,h)\int_{\mathbb R}e^{u(t)}\,dt<+\infty,
$$
because $e^{u(t)}$ decays exponentially as $|t|\to+\infty$.
Thus cylindrical manifolds with nonnegative Ricci curvature support finite-mass
solutions, but these solutions have faster-than-logarithmic decay.
\end{example}

\begin{example}\label{exe3}
The one-dimensional construction above also gives solutions on
$\mathbb R^n$ with infinite mass and non-logarithmic behaviour at infinity. Indeed, let
$U=U(t)$ be the even solution of
$$
-\left(|U'|^{n-2}U'\right)'=e^U\qquad\text{in }\mathbb R
$$
constructed in Example \ref{exe2}. Then
$$
u(x):=U(x_1),\qquad x=(x_1,x')\in\mathbb R\times\mathbb R^{n-1},
$$
is a weak solution of
$$
-\Delta_n u=e^u\qquad\text{in }\mathbb R^n.
$$
Moreover,
$$
U(t)\sim-a|t|\qquad\text{as } |t|\to+\infty
$$
for some $a>0$. Hence $u$ has linear decay in the $x_1$-direction, but it
does not decay along the transverse directions. In particular, $u$ does not
satisfy any radial logarithmic asymptotic profile of the form
$$
u(x)\sim-\gamma\log |x|\qquad\text{as } |x|\to+\infty.
$$
Its mass is infinite, since
$$
\int_{\mathbb R^n}e^u\,dx=\left(\int_{\mathbb R}e^{U(t)}\,dt\right)\left(\int_{\mathbb R^{n-1}}dx'\right)=+\infty\,.
$$
\end{example}

\begin{example}\label{exe4}
We construct complete rotationally symmetric manifolds with
nonnegative Ricci curvature and zero asymptotic volume ratio carrying finite-mass
solutions of
$$
-\Delta_n u=e^u
$$
whose decay is barely faster than logarithmic.

Let $R_0>0$ and $\ell:[R_0,+\infty)\to(0,+\infty)$ be a smooth function such that
$$
\ell(r)\to+\infty,\qquad
\ell(r)=o\big((\log r)^a\big)
\qquad\text{as }r\to+\infty
$$
for every $a>0$, and assume that $\frac{r}{\ell(r)}$
is increasing, concave, and satisfies
$$
0<\left(\frac{r}{\ell(r)}\right)'\leq1
$$
Choose a smooth function
$$
\psi:[0,+\infty)\to[0,+\infty)
$$
such that
$$
\psi(r)=r\quad\text{for }r\leq1,\qquad\psi(r)=\frac{r}{\ell(r)}\quad\text{for }r\geq R_0,
$$
and
$$
0<\psi'(r)\le1,\qquad\psi''(r)\le0\qquad\text{for }r>0.
$$
This can be obtained by a smooth concave interpolation between the Euclidean
warping $r$ near the origin and the eventually concave function $\frac{r}{\ell(r)}$
at infinity.

Consider the rotationally symmetric manifold
$$
(M^n,g)=\left(\mathbb R^n,\,dr^2+\psi(r)^2g_{\mathbb S^{n-1}}\right).
$$
The same argument as in Lemma \ref{lem_wpm} shows that $(M,g)$ is complete and that $\operatorname{Ric}\geq0$.

We first compute the asymptotic volume ratio. We have
$$
\operatorname{Vol}(B_R)=\sigma_{n-1}\int_0^R\psi(r)^{n-1}\,dr,
$$
where $\sigma_{n-1}=|\mathbb S^{n-1}|$. Fix $A>0$. Since $\ell(r)\to+\infty$, there exists $R_A>R_0>0$ such that
$$
\ell(r)\geq A \qquad\text{for }r\geq R_A.
$$
Then we have
$$
\psi(r)\leq \frac{r}{A}\qquad\text{for }r\geq R_A.
$$
Thus, for $R>R_A$,
$$
\operatorname{Vol}(B_R)\leq\sigma_{n-1}\int_0^{R_A}\psi(r)^{n-1}\,dr+\sigma_{n-1}\int_{R_A}^R\left(\frac{r}{A}\right)^{n-1}\,dr
\leq C_A+\frac{\sigma_{n-1}}{nA^{n-1}}R^n.
$$
Dividing by \(\omega_nR^n\), and using \(\sigma_{n-1}=n\omega_n\), we get
$$
\limsup_{R\to+\infty}\frac{\operatorname{Vol}(B_R)}{\omega_nR^n}\leq \frac{1}{A^{n-1}}.
$$
Since $A>0$ is arbitrary we obtain
$$
\operatorname{AVR}(M,g)=0.
$$
We now construct a radially decreasing solution. Arguing as in Lemma \ref{lem_wpm_sol} for a
radial $u=u(r)$ we set
$$
y:=-u',\qquad A(r):=\psi(r)^{n-1}y(r)^{n-1}.
$$
Then the equation
$$
-\Delta_nu=e^u
$$
is equivalent to
$$
\begin{cases}
A'(r)=\psi(r)^{n-1}e^{u(r)},\\
u'(r)=-A(r)^{\frac1{n-1}}\psi(r)^{-1}.
\end{cases}
$$
With the regular initial conditions
$$
A(0)=0,\qquad u(0)=\kappa,
$$
and since $\psi(r)=r$ near $0$, the same local construction as in
Lemma \ref{lem_wpm_sol} gives a regular radial solution near the pole.

The proof of Lemma \ref{lem_wpm_sol} applies verbatim to the present
warping function, because it only uses $0\leq\psi'\leq1$. We obtain
$$
A(R)^{\frac n{n-1}}+\frac n{n-1}\psi(R)^ne^{u(R)}=\frac{n^2}{n-1}\int_0^R\psi^{n-1}\psi'e^u\,dr\leq\frac{n^2}{n-1}\int_0^R\psi^{n-1}e^u\,dr=\frac{n^2}{n-1}A(R),
$$
Thus $A$ is increasing and bounded. Consequently the local solution extends
globally, $u$ is strictly decreasing for $r>0$, and
$$
A(r)\rightarrow L\in(0,+\infty)\quad\text{as }r\rightarrow+\infty,\qquad u(r)\rightarrow -\infty\quad\text{as }r\rightarrow+\infty.
$$
We now compute its decay. We have
$$
u'(r)=-\frac{A(r)^{\frac1{n-1}}}{\psi(r)}\sim-\frac{L^{\frac1{n-1}}}{\psi(r)}=-L^{\frac1{n-1}}\frac{\ell(r)}{r}.
$$
Consequently, we have
$$
u(r)\sim-L^{\frac{1}{n-1}}\Phi(r) \qquad\text{as }r\rightarrow\infty\,,
$$
with
$$
\Phi(r):=\int_{R_0}^r\frac{\ell(s)}{s}\,ds\,.
$$
Note that, since $\ell(r)\rightarrow+\infty$,
$$
\frac{\Phi(r)}{\log r}\rightarrow+\infty\,,
$$
that is $\log r=o(\Phi(r))$. Thus the decay of $u$ is faster than logarithmic. On the other hand, by the monotonicity of $\ell$,
$$
\Phi(r)=\int_{R_0}^r\frac{\ell(s)}{s}\,ds\leq\ell(r)\int_{R_0}^r\frac{ds}{s}=\ell(r)\log r+O(\ell(r))\,.
$$
Thus $u$ decays faster than logarithmically, but only barely so.

Note that if we further assume
$$
\frac{\ell'(r)}{\ell(r)}=o\left(\frac{1}{r\log r}\right)\qquad\text{as }r\rightarrow\infty
$$
we have
$$
\lim_{r\rightarrow\infty}\frac{\Phi(r)}{\ell(r)\log r}=\lim_{r\rightarrow\infty}\frac{\frac{\ell(r)}{r}}{\frac{\ell(r)}{r}+\ell'(r)\log r}=1
$$
i.e. $u(r)\sim-L^{\frac{1}{n-1}}\Phi(r)\sim-L^{\frac{1}{n-1}}\ell(r)\log r$ as $r$ tends to infinity. Note that
$$
\ell(r)=\big(\underbrace{\log\log\cdots\log r}_{k\text{ times}}\big)^\beta,
$$
satisfies all the above assumptions for $k=2$, $0<\beta<1$ or $k\geq3$, $\beta>0$.

Finally, the solution has finite mass. Indeed,
$$
\int_M e^u\,dV_g=\sigma_{n-1}\int_0^\infty e^{u(r)}\psi(r)^{n-1}\,dr=\sigma_{n-1}\int_0^\infty A'(r)\,dr=\sigma_{n-1}L<+\infty\,.
$$
\end{example}

\

\appendix

\section{A Liouville-type theorem}\label{app:karp}

We show the following improvement {\em \`a la Karp} of a classical Liouville-type theorem due to Berestycki, Caffarelli and Nirenberg \cite{BCN}.
\begin{lemma}\label{lem_karp} Let $(M,g)$ be a complete noncompact Riemannian manifold of dimension $n\geq 2$. Let $h,E\in L^1_{\text{loc}}(M)$ be nonnegative functions and $T\in L^1_{\text{loc}}(TM)$ be a locally integrable vector field such that
\begin{equation}\label{5}
  \operatorname{div}T\geq E
\end{equation}
in the distributional sense and
\begin{equation}\label{6}
|T|^2\leq C_0hE
\end{equation}
a.e. on $M$ for some positive constant $C_0>0$. Let $F:[1,\infty)\rightarrow(0,\infty)$ be a positive nondecreasing function such that
\begin{equation}\label{9}
\int_{R_0}^\infty\frac{1}{tF(t)}\,dt=+\infty\,.
\end{equation}
If for some $o\in M$
\begin{equation}\label{8}
\limsup_{R\rightarrow\infty}\frac{1}{R^2F(R)}\int_{B_R(o)\setminus B_\frac{R}{2}(o)}h\,dV_g<\infty
\end{equation}
then $E=0$ on $M$.
\end{lemma}

\begin{proof}
We follow the proof of Karp \cite[Theorem 2.2]{karp}. We argue by contradiction and assume that $E\not\equiv0$. Since
$E\ge0$ and $E\in L^1_{\mathrm{loc}}(M)$, there exists $R_*>R_0$ such that
$$
\int_{B_{R_*}(o)}E\,dV_g>0.
$$
For $0<s<t$, let $\eta=\eta_{s,t}$ be a standard
smooth cut-off function satisfying
$$
0\le \eta\le1,\qquad
\eta\equiv1\ \text{on }B_s(o),\qquad
\eta\equiv0\ \text{on }M\setminus B_t(o),
$$
and
$$
|\nabla\eta|\le \frac{C}{t-s}\, ,
$$
for some constant $C>0$ independent of $s,t,o$. Testing \eqref{5} with $\eta^2\ge0$ and using \eqref{6} we obtain
\begin{align*}
\int_M E\eta^2\,dV_g&\le
-2\int_M \eta\langle T,\nabla\eta\rangle\,dV_g\\
&\le
2\int_{\operatorname{supp}\nabla\eta}
\eta |T|\,|\nabla\eta|\,dV_g\\
&\le
2C_0^{1/2}
\left(
\int_{\operatorname{supp}\nabla\eta} E\eta^2\,dV_g
\right)^{1/2}
\left(
\int_{\operatorname{supp}\nabla\eta} h|\nabla\eta|^2\,dV_g
\right)^{1/2}.
\end{align*}
Therefore
\begin{equation}\label{7}
\left(\int_M E\eta^2\,dV_g\right)^2
\le\frac{C}{(t-s)^2}\left(\int_{B_t(o)\setminus B_s(o)} h\,dV_g\right)\left(\int_{B_t(o)\setminus B_s(o)} E\eta^2\,dV_g\right).
\end{equation}
Let $R_j=2^jR_0$, with $j$ large enough so that $R_j\ge R_*$, set $s=R_{j}$, $t=R_{j+1}$ and let $\eta_j$ be such that
$$
0\leq\eta_j\leq\eta_{j+1}\leq1,\qquad
\eta_j\equiv1 \ \text{on }B_{R_{j}}(o),\qquad
\eta_j\equiv0 \ \text{on }M\setminus B_{R_{j+1}}(o),
$$
and
$$
|\nabla\eta_j|\le \frac{C}{R_{j+1}-R_{j}}
= \frac{C}{R_{j}}.
$$
Let
$$
Q_j:=\int_ME\eta_j^2\,dV_g=\int_{B_{R_{j+1}}(o)}E\eta_j^2\,dV_g.
$$
Then $Q_j>0$ for all large $j$, $Q_j$ is nondecreasing and from \eqref{7} we get for large $j$'s
$$
Q_{j}^2=\left(\int_M E\eta_{j}^2\,dV_g\right)^2\leq
C\left(
\frac{1}{R_{j}^2}
\int_{B_{R_{j+1}}(o)\setminus B_{R_{j}}(o)}h\,dV_g
\right)
\left(\int_{B_{R_{j+1}}(o)\setminus B_{R_j}(o)} E\eta_j^2\,dV_g\right).
$$
By \eqref{8} we have
$$
\frac{1}{R_{j}^2}\int_{B_{R_{j+1}}(o)\setminus B_{R_{j}}(o)}h\,dV_g\leq C F(R_{j+1})
$$
for all large $j$. Thus, since $Q_j$ is nondecreasing in $j$, we obtain
$$
Q_jQ_{j-1}\leq Q_{j}^2\le C F(R_{j+1})\left(Q_{j}-Q_{j-1}\right)
$$
and
$$
\frac{1}{F(R_{j+1})}\leq C\left(\frac{1}{Q_{j-1}}-\frac{1}{Q_{j}}\right).
$$
Then we get
\begin{equation}\label{10}
\sum_{j=j_0}^{+\infty}\frac{1}{F(R_j)}<+\infty .
\end{equation}
Since $F$ is nondecreasing and $R_j=2^jR_0$, we have
$$
\sum_{j=j_0}^\infty \frac{1}{F(R_j)}=\sum_{j=j_0}^\infty \int_{R_j}^{R_{j+1}}\frac{1}{R_jF(R_j)}\,dt
\geq\sum_{j=j_0}^\infty \int_{R_j}^{R_{j+1}}\frac{1}{tF(t)}\,dt=\int_{R_{j_0}}^\infty\frac{1}{tF(t)}\,dt\,.
$$
This and \eqref{10} contradict \eqref{9}. Therefore $E\equiv0$ on $M$.
\end{proof}

As an application of the previous lemma, we obtain the following weighted
Karp-type result, which is not used in this paper but may be of independent
interest.

\begin{corollary}\label{cor_karp}
Let $(M,g)$ be a complete noncompact Riemannian manifold.
Let $a\ge 0$ be a measurable locally integrable weight and let
$\sigma\in H^1_{\mathrm{loc}}(M)$ be such that
$$
a\sigma^2\in L^1_{\mathrm{loc}}(M), \qquad a|\nabla\sigma|^2\in L^1_{\mathrm{loc}}(M), \qquad a\sigma\nabla\sigma\in L^1_{\mathrm{loc}}(TM).
$$
Assume that
$$
\sigma\,\operatorname{div}(a\nabla\sigma)\ge 0
$$
in the sense of distributions. Let $F:[R_0,+\infty)\to(0,+\infty)$ be positive and nondecreasing, and assume that
$$
\int_{R_0}^{+\infty}\frac{dt}{tF(t)}=+\infty .
$$
If there exist $o\in M$ and a constant $C>0$ such that
\begin{equation}\label{11}
\int_{B_{2R}(o)\setminus B_R(o)} a\sigma^2\,dV_g \le C R^2F(R)
\end{equation}
for every $R\ge R_0$, then
$$
a|\nabla\sigma|^2\equiv 0
$$
on $M$. In particular, if $a>0$ almost everywhere and $M$ is connected, then
$\sigma$ is constant.
\end{corollary}

\begin{proof}
We apply Lemma \ref{lem_karp} with
$$
T:=a\sigma\nabla\sigma, \qquad E:=a|\nabla\sigma|^2, \qquad h:=a\sigma^2.
$$
Indeed, since $\sigma\,\operatorname{div}(a\nabla\sigma)\ge 0$
in the sense of distributions, we have
$$
\operatorname{div}T=\operatorname{div}(a\sigma\nabla\sigma)=a|\nabla\sigma|^2+\sigma\operatorname{div}(a\nabla\sigma)\ge a|\nabla\sigma|^2=E
$$
in the sense of distributions. Moreover,
$$
|T|^2=a^2\sigma^2|\nabla\sigma|^2=hE.
$$
Finally \eqref{8} immediately follows from the definition of $h$ and \eqref{11}. Therefore Lemma \ref{lem_karp} gives $E\equiv0$, namely
$
a|\nabla\sigma|^2\equiv0.
$
If $a>0$ almost everywhere and $M$ is connected, it follows that
$|\nabla\sigma|=0$ almost everywhere on $M$, hence $\sigma$ is constant.
\end{proof}

\

\begin{ackn}
\noindent The authors are grateful to G. Ciraolo, A. Farina and M. Gatti for pointing out a gap in the proof of a capacitary estimate in an earlier version of the manuscript.

The first and third authors are member of the {\em GNSAGA, Gruppo Nazionale per le Strutture Algebriche, Geometriche e le loro Applicazioni} of INdAM. The second author is member of {\em GNAMPA, Gruppo Nazionale per l'Analisi Matematica, la Probabilit\`a e le loro Applicazioni} of INdAM.
\end{ackn}

\

\medskip

\subsection*{Data availability statement}
Data sharing not applicable to this article as no datasets were generated or analysed during the current study.

\medskip

\subsection*{Conflicts of interests/Competing interests} The authors have no conflicts of interest to declare that are relevant to the content of this article.

\medskip

\subsection*{Funding} No funding was received to assist with the preparation of this manuscript.

\

\

\

\

\


\begin{thebibliography}{99}

\bibitem{AnCiFa} C.A. Antonini, G. Ciraolo, A. Farina. {\em Interior regularity results for inhomogeneous anisotropic quasilinear equations.} Math. Ann. 387 (2023), 1745--1776.

\bibitem{BaKr} Z. M. Balogh, A. Krist{\'a}ly. {\em Sharp isoperimetric and Sobolev inequalities in spaces with nonnegative Ricci curvature.} Math. Ann. 385 (2023),  1747--1773.

\bibitem{BCN} H. Berestycki, L. Caffarelli, L. Nirenberg. {\em Further qualitative properties for elliptic equations in unbounded domains.} Ann. Scuola Norm. Sup. Pisa Cl. Sci. (4) 25 (1997), 69--94.


\bibitem{Bre} S. Brendle. {\em Sobolev inequalities in manifolds with nonnegative curvature.} Comm. Pure Appl. Math. 76(9) (2023), 2192--2218.



\bibitem{CaffarelliGidasSpruck}
L.A. Caffarelli, B. Gidas, J. Spruck.
{\em Asymptotic symmetry and local behavior of semilinear elliptic equations
with critical Sobolev growth.}
Comm. Pure Appl. Math. 42 (1989), no. 3, 271--297.

\bibitem{CaiLai} X. Cai, M. Lai. \emph{Liouville equations on complete surfaces with nonnegative Gauss curvature.} Pacific J. Math. 332(2024), no. 1, 23--37.



\bibitem{CatLiMonRon}
G. Catino, Y.Y. Li, D.D. Monticelli, A. Roncoroni. {\em A Liouville theorem in the {Heisenberg} group.} To appear, J. Eur. Math. Soc.

\bibitem{CM}  G. Catino, D.D. Monticelli. \emph{Semilinear elliptic equations on manifolds with nonnegative Ricci curvature.} J. Eur. Math. Soc. 28 (2026), no. 1, pp. 359--392.




\bibitem{CMR}
G. Catino, D.D. Monticelli, A. Roncoroni.
{\em On the critical $p$-Laplace equation.}
Adv. Math. 433 (2023), Paper No. 109331.


\bibitem{CatMonRon2}
  G. Catino, D.D. Monticelli, A. Roncoroni. {\em Rigidity of solutions to singular/degenerate semilinear critical equations.} J. Funct. Anal. 291 (2026), 111538.


\bibitem{CatMonRonWang}
G. Catino, D.D. Monticelli, A. Roncoroni, X. Wang. {\em Liouville Theorems on pseudohermitian manifolds with nonnegative {Tanaka--Webster} curvature.} Preprint, arxiv:2412.08500.

\bibitem{ChenLi}
W. Chen, C. Li.
{\em Classification of solutions of some nonlinear elliptic equations.}
Duke Math. J. 63 (1991), no. 3, 615--622.

\bibitem{CC} G. Ciraolo, R. Corso. \emph{Symmetry for positive critical points of Caffarelli-Kohn-Nirenberg inequalities.} Nonlinear Anal. 216 (2022) 112683.

\bibitem{CirEspLi}
G. Ciraolo, P. Esposito, X. Li. {\em On the classification of solutions to a class of \(N\)-Liouville equations in \(\mathbb R^N\).} Preprint, arxiv:2604.10050.


\bibitem{CFP} G. Ciraolo, A. Farina, C.C. Polvara. \emph{Classification results, rigidity theorems and semilinear PDEs on Riemannian manifolds: A P-function approach.} J. Eur. Math. Soc. (2025), published online first.

\bibitem{CFR} G. Ciraolo, A. Figalli, A. Roncoroni. \emph{Symmetry results for critical anisotropic $p-$Laplacian equations in convex cones.} Geom. Funct. Anal. 30 (2020) 770--803.

\bibitem{CirPol}  G. Ciraolo, C.C. Polvara. {\em On the classification of extremals of Caffarelli-Kohn-Nirenberg inequalities.} Calc. Var. Partial Differential Equations 64 (2025), no. 8, 246.


\bibitem{CiraoloLi}
G. Ciraolo, X. Li.
{\em Classification of solutions to the anisotropic $N$-Liouville equation
in $\mathbb R^N$.}
Int. Math. Res. Not. IMRN 2024 (2024), no. 19, 12824--12856.

\bibitem{CoVo} S. Cohn-Vossen, {\em K\"{u}rzeste Wege und Totalkr\"{u}mmung auf Fl\"{a}chen.} Compositio Math. 2(1935), 69--133.



\bibitem{DamascelliMerchanMontoroSciunzi}
L. Damascelli, S. Merch\'an, L. Montoro, B. Sciunzi.
{\em Radial symmetry and applications for a problem involving the
$-\Delta_p(\cdot)$ operator and critical nonlinearity in $\mathbb R^n$.}
Adv. Math. 265 (2014), 313--335.


\bibitem{DiB} E. DiBenedetto. {\em $C^{1+\alpha}$ local regularity of weak solutions of degenerate elliptic equations.} Nonlin. An. 7(8) (1983), 827--850.

\bibitem{Esp} P. Esposito. {\em A classification result for the quasi-linear {Liouville} equation.} Ann. Inst. H. Poincar\'e C Anal. Non Lin\'eaire 35(3) (2018), 781--801.



\bibitem{EspositoLucia}
P. Esposito, M. Lucia.
{\em Harnack inequalities and quantization properties for the $n$-Liouville
equation.}
Calc. Var. Partial Differential Equations 63 (2024), no. 6,
Paper No. 159, 17 pp.

\bibitem{FlyVet}
J. Flynn, J. V{\'e}tois. {\em Liouville-type results for the {CR} Yamabe equation in the {Heisenberg} group.} Ann. Sc. Norm. Super. Pisa Cl. Sci., to appear, arxiv:2310.14048.



\bibitem{FMM} M. Fogagnolo, A. Malchiodi, L. Mazzieri, {\em A note on the critical Laplace Equation and Ricci curvature.} J. Geom. Anal. 33 (2023), no. 6, 1--17.

\bibitem{Hub} A. Huber. {\em On subharmonic functions and differential geometry in the large.} Comment. Math. Helv. 32 (1957), 13--72.


\bibitem{JL} D. Jerison, J.M. Lee. \emph{Extremals for the Sobolev inequality on the Heisenberg group and the CR Yamabe problem.} J. Am. Math. Soc. 1 (1988) 1--13.

\bibitem{karp} L. Karp. \emph{Subharmonic Functions on Real and Complex Manifolds.} Math. Z. 179, 535--554 (1982).

\bibitem{LiZhang} Y.Y. Li, L. Zhang, {\em Liouville-type theorems and Harnack inequalities for semilinear elliptic
equations.} J. Anal. Math. 90 (2003), 27--87.


\bibitem{Oup} Q. Ou. \emph{ On the classification of entire solutions to the critical $p$-Laplace equation}. Math. Ann. 392 (2025), 1711--1729.

\bibitem{Ou} Q. Ou.  \emph{Classification results of Liouville equations and rigidity of Riemannian surfaces.} Preprint, arXiv:2604.27973.

\bibitem{pet}
P. Petersen,
\emph{Riemannian Geometry},
Graduate Texts in Mathematics, vol.~171,
3rd ed., Springer, Cham, 2016.



\bibitem{Sciunzi}
B. Sciunzi.
{\em Classification of positive $D^{1,p}(\mathbb R^n)$-solutions to the
critical $p$-Laplace equation in $\mathbb R^n$.}
Adv. Math. 291 (2016), 12--23.

 \bibitem{Sun-Wang} L. Sun, Y. Wang. {\em Critical quasilinear equations on Riemannian manifolds.} Preprint.

\bibitem{tashiro}
Y. Tashiro.
{\em Complete Riemannian manifolds and some vector fields.}
Trans. Amer. Math. Soc. 117 (1965), 251--275.

\bibitem{Tol} P. Tolksdorf. {\em Regularity for a more general class of quasilinear elliptic equations.} J. Diff. Eq. 51(1) (1984), 126--150.



\bibitem{Vetois}
J. V\'etois.
{\em A priori estimates and application to the symmetry of solutions for
critical $p$-Laplace equations.}
J. Differential Equations 260 (2016), no. 1, 149--161.

\bibitem{Vetois2} J. V\'etois. {\em A note on the classification of positive solutions to the critical $p$-Laplace equation in $\mathbb{R}^n$.} Adv. Nonlinear Stud. 24 (2024), no. 3, 543--552.






\end{thebibliography}
\end{document}